\documentclass[reqno]{article}
\usepackage[top=3cm, bottom=2.5cm, left=3cm, right=3cm]{geometry}
\usepackage{lmodern}

\usepackage{amssymb}
\usepackage{amsmath}
\usepackage{mathtools}
\usepackage{amsthm}
\usepackage{graphicx}
\usepackage{array}
\usepackage{titlesec}
\usepackage{eepic}
\usepackage{multicol}
\usepackage[usenames,dvipsnames]{xcolor}
\usepackage{color}
\usepackage[linktocpage=true]{hyperref}
\usepackage[hypcap=true]{caption}
\usepackage[cmtip,all]{xy}
\usepackage{enumitem}
\usepackage{leftidx}
\usepackage[dvipsnames]{xcolor}
\usepackage[mathscr]{euscript}
\usepackage{setspace}
\usepackage{needspace}

\usepackage{parskip}

\makeatletter
\def\thm@space@setup{%
  \thm@preskip=\parskip \thm@postskip=0pt
}
\makeatother

\allowdisplaybreaks

\usepackage{titlesec}
\titleformat{\section}[block]{\color{black}\large\bfseries\filcenter}{\thesection.}{0.5em}{}
\titleformat{\subsection}[hang]{\bfseries}{}{0.5em}{}

\numberwithin{equation}{section}

\newtheorem{theorem}{Theorem}[section]
\newtheorem{lemma}[theorem]{Lemma}
\newtheorem{proposition}[theorem]{Proposition}
\newtheorem{corollary}[theorem]{Corollary}
\newtheorem{definition}[theorem]{Definition}
\newtheorem{problem}[theorem]{Problem}

\newtheorem{remark}[theorem]{Remark}

\newtheorem{Atheorem}{Theorem}

\newtheorem{Acorollary}[Atheorem]{Corollary}

\newtheorem{Btheorem}{Theorem}

\newtheorem{Bconjecture}[Btheorem]{Conjecture}

\newtheorem{Ctheorem}{Theorem}

\newtheorem{Cproblem}[Ctheorem]{Problem}

\titleformat{\subsection}[runin]{\bfseries}{}{}{}[.]
\titleformat{\subsubsection}[runin]{\bfseries}{}{}{}[.]

\makeatletter
\renewenvironment{proof}[1][\proofname]{%
   \par\pushQED{\qed}\normalfont%
   \topsep6\p@\@plus6\p@\relax
   \trivlist\item[\hskip\labelsep\bfseries#1\@addpunct{.}]%
   \ignorespaces
}{%
   \popQED\endtrivlist\@endpefalse
}
\makeatother

\addtocounter{tocdepth}{0}
\newcommand{\CC}{\mathbf{C}}

\newcommand{\EE}{\mathbf{E}}
\newcommand{\FF}{\mathbf{F}}

\newcommand{\NN}{\mathbf{N}}

\newcommand{\RR}{\mathbf{R}}
\newcommand{\TT}{\mathbf{T}}

\newcommand{\ZZ}{\mathbf{Z}}

\newcommand{\C}{\mathcal{C}}
\newcommand{\T}{\mathcal{T}}

\newcommand{\D}{\mathcal{D}}
\newcommand{\F}{\mathcal{F}}

\def\S{\mathcal{S}}

\DeclareMathOperator*\baulim{\mathrm{bau-lim}}

\DeclareMathOperator*\spn{\mathrm{span}}

\DeclareMathOperator*\wstspan{\overline{\mathrm{span}^{\wast}}}

\def\Xint#1{\mathchoice
{\XXint\displaystyle\textstyle{#1}}%
{\XXint\textstyle\scriptstyle{#1}}%
{\XXint\scriptstyle\scriptscriptstyle{#1}}%
{\XXint\scriptscriptstyle\scriptscriptstyle{#1}}%
\!\int}
\def\XXint#1#2#3{{\setbox0=\hbox{$#1{#2#3}{\int}$ }
\vcenter{\hbox{$#2#3$ }}\kern-.6\wd0}}

\def\dashint{\Xint-}

\def\1{\mathbf{1}}
\def\Id{\mathrm{id}}
\def\H{{H}}

\def\M{\mathcal{M}}

\def\N{\mathcal{N}}

\newcommand{\loc}{\mathrm{loc}}

\newcommand{\vertiii}[1]{
 {\left\vert\kern-0.25ex\left\vert\kern-0.25ex\left\vert #1 
  \right\vert\kern-0.25ex\right\vert\kern-0.25ex\right\vert}
}

\def\Prj{\mathscr{P}}
\def\B{\mathcal{B}}

\def\L{\mathcal{L}}
\def\wast{{{\mathrm{w}}^\ast}}

\def\cb{\mathrm{cb}}

\def\Tr{\mathrm{Tr}}

\def\R{\mathcal{R}}

\def\algtensor{\otimes_{\mathrm{alg}}}

\def\weaktensor{ \, \overline{\otimes} \, }

\title{
 Noncommutative strong maximals and \\
  almost uniform convergence in several directions
}

\author{
  Jos\'e M. Conde-Alonso,
  Adri\'an M. Gonz\'alez-P\'erez\thanks{Supported by the European
    Research Council. Consolidator Grant \texttt{614195}
    - RIGIDITY.
  } \ and 
  Javier Parcet\thanks{Partially supported by 
    CSIC Grant \texttt{PIE-201650E030} and ICMAT Severo Ochoa Grant \texttt{SEV-2015-0554}. \hfill \null \hskip70pt \null \hskip12pt \textbf{Keywords.} von Neumann algebras; noncommutative $L_p$-martingales; ergodic mean; maximal operators.
  } 
}
\date{}

\begin{document}

\maketitle

\vskip-25pt 

\null

\begin{abstract}
Our first result is a noncommutative form of Jessen/Marcinkiewicz/Zygmund theorem for the maximal limit of multiparametric martingales or ergodic means. It implies bilateral almost uniform convergence (a noncommutative analogue of a.e. convergence) with initial data in the expected Orlicz spaces. A key ingredient is the introduction of the $L_p$-norm of the $\limsup$ of a sequence of operators as a localized version of a $\ell_\infty/c_0$-valued $L_p$-space. In particular, our main result gives a strong $L_1$-estimate for the $\limsup$ |as opposed to the usual weak $L_{1,\infty}$-estimate for the $\sup$| with interesting consequences for the free group algebra.  

\vskip3pt
   
Let $\L \FF_2$ denote the free group algebra with $2$ generators and consider the free Poisson semigroup generated by the usual length function. It is an
open problem to determine the largest class inside $L_1(\L \FF_2)$ for which the free Poisson semigroup converges to the initial data. Currently, the best          known result is $L \log^2 L(\L \FF_2)$. We improve this result by adding to it the operators in $L_1(\L \FF_2)$ spanned by words without signs changes. Contrary to other related results in the literature, this set grows exponentially with length. The proof relies on our estimates for the noncommutative $\limsup$ together with new transference techniques.
  
\vskip3pt

We also establish a noncommutative form of C\'ordoba/Feffermann/Guzm\'an inequality for the strong maximal. More precisely, a weak $(\Phi,\Phi)$ inequality |as opposed to weak $(\Phi,1)$| for noncommutative multiparametric martingales and $\Phi(s) = s (1 + \log_+ s)^{2 + \varepsilon}$. This logarithmic power is an $\varepsilon$-perturbation of the expected optimal one. The proof combines a refinement of Cuculescu's construction with a quantum probabilistic interpretation of M. de Guzm\'an's original argument. The commutative form of our argument gives the simplest known proof of this classical inequality. A few interesting consequences are derived for Cuculescu's projections.
\end{abstract}

\vskip-20pt 

\null

\section*{\bf Introduction} 

\vskip-10pt 

Given $f \in L_1(\RR^n)$, Lebesgue's differentiation theorem gives 
\[
\lim_{r \to 0} A_rf(x) := \lim_{r \to 0} \frac1{|B_r(x)|} \int_{B_r(x)} f(y) \, d y = f(x) \quad \mbox{a.e. $x \in \RR^n$.}
\]
In one dimension, this is a general form of the fundamental theorem of calculus. Almost everywhere convergence to original data is also known to hold in other scenarios when replacing Euclidean balls by other averaging processes. In probability theory, we use instead conditional expectations onto martingale filtrations of $\sigma$-subalgebras. In ergodic theory, a similar result holds as well for ergodic means or subordinate Markovian semigroups. In all these settings, the strategy reduces to prove a \emph{maximal inequality}. That is, a quantitative estimate for the maximal operator $Mf(x) = \sup_{\alpha} |A_\alpha(f)|$ where $(A_\alpha)_{\alpha}$ is an averaging process, as those described above. The underlying maximal inequality behind Lebesgue differentiation is the celebrated Hardy-Littlewood maximal theorem \cite{HardyLil1930Max}. Almost everywhere convergence of martingale approximations relies on Doob's maximal inequality \cite{Doob1953stochastic}. In ergodic theory, maximal inequalities for ergodic means and Markovian semigroups where respectively established by Yosida/Kakutani \cite{YosidaKakutani1939, YosidaKakutani1941} and Dunford/Schwartz \cite{DunfordSchwartz1956}. 

In all the scenarios laid down before, the maximal operator is bounded as a map $L_1 \to L_{1,\infty}$. That is usually referred to as an operator of weak type $(1,1)$. In particular, we obtain almost everywhere convergence $A_\alpha(f) \to f$ for every $f \in L_1$, which is best possible in the context of one single parameter $\alpha$. The problem above makes sense for multiple directions though. Indeed, given a function $f \in L_1(\RR^2)$ one can ask whether the two-parameter averages
\[
  (A_n \otimes A_m)f(x,y)
  \, = \,
   \frac{n}{2} \, \frac{m}{2} \, \int_{x - \frac1{n}}^{x + \frac1{n}} \int_{y - \frac1{m}}^{y + \frac1{m}} f(s,t) \, d s \, d t
  \, = \,
  \dashint_{x - \frac1{n}}^{x + \frac1{n}} \dashint_{y - \frac1{m}}^{y + \frac1{m}} f(s,t) \, d s \, d t
\]
converge a.e. to $f$ as $n,m \to \infty$. This is easy to prove when $f$ is locally in $L_p$ for some $p>1$, since the associated maximal operator |usually referred to as the \emph{strong maximal operator}| is bounded from $L_p$ to $L_p$ by a Fubini type argument and Marcinkiewicz interpolation. This technique fails in the quasi-Banach space $L_{1,\infty}$, where tensor product arguments become inefficient since one can produce simple tensors in $L_{1,\infty}(\RR) \algtensor L_{1,\infty}(\RR)$ that are not in $L_{1,\infty}(\RR^2)$. This difficulty is not an artifice of the proof. Indeed, in 1933, Saks constructed the first integrable function $f \in L_1(\RR^2)$ for which $(A_n \otimes A_m)f$ fails to converge a.e. to $f$. The multiparametric problem was thus reoriented towards finding the largest class inside $L_1$ for which there is almost everywhere convergence in several directions. The first positive result was obtained by Jessen, Marcinkiewicz and Zygmund in \cite{Jessen1935}, where they proved that $L \log^{d-1} L$ suffices in $d$ parameters. They also showed that this space was optimal in terms of Orlicz classes. Their main result for $d=2$ states that, replacing the $\sup$ by a $\limsup$ in the definition of maximal operator, gives a sublinear map $L \log L \to L_1$
\begin{equation}
  \label{eq:classicalJMZ}
  \Big\| \limsup_{n,m \to \infty} \big|(A_n \otimes A_m)f\big| \Big\|_1
  \, \lesssim \,
  \| f \|_{L \log L}.
  \tag{JMZ}
\end{equation}
The crucial point here is that replacing $\sup$'s by $\limsup$'s allows for strong $L_1$-estimates (instead of weak ones) where Fubini-type arguments are available. This is not weaker nor stronger than the usual maximal $L_{1,\infty}$-estimates and it still yields a.e. convergence to initial data. Moreover, their functional-analytic approach is very flexible and their argument works verbatim for multiparametric martingales and ergodic means. Later, C\'ordoba/Fefferman and independently M. de Guzm\'an proved in \cite{CorFeff1975,Guzman1972ProductBases} the following weak type inequality for the strong maximal, whose restricted type form goes back to \cite{Jessen1935}. For every $\lambda > 0$ and positive $f \in L \log L$, it holds that
\begin{equation}
  \label{eq:classicalCF}
  \Big| \Big\{ \sup_{n, m \ge 1} (A_n \otimes A_m)f(x,y) > \lambda \Big\} \Big|
  \, \lesssim \,
  \int \hskip-5pt \int \frac{f}{\lambda} \Big( 1 + \log_+ \Big(\frac{f}{\lambda} \Big) \Big) \, d x \, d y.
  \tag{CFG}
\end{equation}
Inequalities \eqref{eq:classicalJMZ} and \eqref{eq:classicalCF} extend to $d$ variables
just by working with the Orlicz spaces $L \log^{d - 1} L$.
%We also mention that the inequality \eqref{eq:classicalCF} works in the context of martingales \cite{Walsh1979StrongMartingales, Frangos1988} and with respect to more abstract differentiation bases, see \cite{Guzman1975}.

\vskip-20pt

\null

% nc stuff
\subsection*{\bf Noncommutative maximal inequalities} The theory of von Neumann algebras extends measure theory to noncommutative analogues of $L_\infty$-spaces. Given the local nature of the results above, we shall be concerned with quantum probability spaces. In other words, with pairs $(\M,\tau)$ given by a finite von Neumann algebra $\M$ with a normalized trace $\tau$. Many measure-theoretical and probabilistic notions admit a natural analogue in this setting. Noncommutative $L_p$-spaces, noncommutative martingales and ergodic analogues will be of particular interest for us. Avoiding by now precise definitions for these objects, the strategy remains to prove \lq\lq almost everywhere convergence" of certain averaging processes by establishing appropriate \lq\lq maximal inequalities" in optimal endpoint spaces. The lack of points after quantization forces though to introduce more involved notions for the noncommutative analogues of the expressions just quoted. The notion of almost everywhere convergence admits natural generalizations that go back at least to Lance \cite{Lance1976Ergodic}. Egorov's theorem in a classical probability space $(\Omega,\mu)$ gives that $f_n \to f$ a.e. iff there exist measurable sets $E$ of arbitrarily small measure such that $(f - f_n) \1_{\Omega \setminus E} \to 0$ uniformly. This alternative form of convergence is usually called \emph{almost uniform convergence}. It generalizes to von Neumann algebras replacing the sets $E$ by  projections $e \in \M$ of arbitrarily small trace. By noncommutativity, the left and/or right position of $\1 - e$ with respect to the sequence $f-f_n$ matters. The bilateral choice $(\1 - e) (f - f_n) (\1 - e)$ is the right one with original data $f \in L_1 \setminus L_2$, as we shall recall in the body of the paper. 

The formulation of noncommutative maximal inequalities is already subtle, since it is not possible to define the supremum of a family of noncommuting operators in a meaningful way. Namely, as in the Introduction of \cite{JunXu2007}, one can produce examples of $2 \times 2$ noncommuting matrices $A_1, A_2, A_3$ for which no $2 \times 2$ matrix $A$ satisfies $\langle \xi, A \xi \rangle = \max \{ \langle \xi, A_j \xi \rangle: j=1,2,3 \}$ for all $\xi \in \RR^2$. The trick to overcome this for weak type maximal inequalities is to use $$\Big\{ \sup_{n \ge 1} f_n > \lambda \Big\} = \Omega \setminus \bigcap_{n \ge 1} \Big\{ f_n \le \lambda \Big\}.$$ That is, the $\lambda$-level set of the maximal function is the complement of the set where all entries $f_n$ are bounded above by $\lambda$. An analogous formulation is possible for von Neumann algebras. Let $\Prj(\M)$ be the lattice of projections in $\M$. Given $f \in L_1(\M)_+$ and $\lambda > 0$, the goal is to find a projection $q(\lambda) \in \Prj(\M)$ such that 
\begin{enumerate}[leftmargin=1.2cm, label={\rm (\roman*)}, ref={\rm (\roman*)}]
  \item $q(\lambda) \, A_n(f) \, q(\lambda) \, \leq \, \lambda \, q(\lambda)$.
  \item $\displaystyle{
          \tau \big( q(\lambda)^\perp \big) = \tau \big( \1 - q(\lambda) \big)
          \, \lesssim \, 
          \frac{\| f \|_1}{\lambda}.
        }$
\end{enumerate}
The noncommutative form of Yosida-Kakutani maximal theorem for ergodic averages was obtained by Yeadon \cite{Yeadon1977,Yeadon1980II} following earlier results of Lance \cite{Lance1976Ergodic}. The noncommutative endpoint for Doob's maximal inequality was proved by Cuculescu in his seminal paper \cite{Cuculescu1971}. As in the classical case, these inequalities can be used to prove that $A_n(f) \to f$ bilaterally almost uniformly with initial data in $L_1(\M)$. It is known to experts though that one should regard these inequalities as \lq\lq weakened forms" of their classical counterparts. As a byproduct of our results, we shall be able to quantify to what extend this happens in the context of Cuculescu's construction of $q(\lambda)$.   

% Lp theory
By Marcinkiewicz interpolation, the alluded maximal operators are $L_p$-bounded over classical measure spaces $(\Omega,\mu)$ for all $p > 1$. It took however more than 25 years to get noncommutative $L_p$-maximal inequalities in the line of Cuculescu and Yeadon theorems. Indeed, the $L_p(\Omega)$-norm of a maximal function $Mf = \sup_n f_n$ must be understood as the $L_p(\Omega;\ell_\infty)$-norm of the sequence $(f_n)_n$ to be meaningful in noncommutative algebras. The \emph{mixed-norm spaces} $L_p(\M;\ell_\infty)$ where introduced by Pisier for hyperfinite algebras \cite{Pi1998} |later by Junge in full generality| and required the full strength of operator space theory. It was Junge who extended in 2002 
Doob's $L_p$-maximal inequality for noncommutative martingales with an ad hoc argument heavily relying in Hilbert module theory and duality \cite{Jun2002Doob}. A few years later, Junge and Xu obtained $L_p$-maximal inequalities for ergodic means and subordinated Markovian semigroups \cite{JunXu2007}. The key point was a novel interpolation theorem for families of positive preserving maps |as a substitute for Marcinkiewicz interpolation| which allows to infer results from Cuculescu and Yeadon \lq\lq extra-weak" inequalities. The subtle price for this is that the interpolation constants grow like $(p-1)^{-2}$ as $p \to 1^+$, the square of the classical growth rate. This different quantitative behavior translates into qualitative properties. Indeed, using a noncommutative version of Yano's extrapolation, it was shown in \cite{Hu2009} that $(A_n)_n: L \log^2 L(\M) \to L_1(\M;\ell_\infty)$, where the extra exponent in the logarithm is directly connected to the extra exponent in the $L_p$-operator norm. Both quantitative and qualitative results above are the best possible. Again, one could argue in hindsight that such deviations from the classical results arise from the extra-weak nature of Cuculescu's and Yeadon's projections. We shall find additional evidence of this here. 

Our primary goal in this paper is to investigate noncommutative strong maximals, almost uniform convergence in several directions to initial data and applications in noncommutative harmonic analysis. More precisely, we shall prove noncommutative forms of Jessen/Marcinkiewicz/Zygmund theorem \eqref{eq:classicalJMZ} and C\'ordoba/Fefferman/Guzm\'an inequality \eqref{eq:classicalCF} for quite general processes. Beyond the (expected) tensor products of martingale filtrations, ergodic means or subordinated Markov semigroups, our examples include free products and other (noncommuting) compositions. A result of independent interest in the free group algebra will also be proved. Namely, combining our techniques towards noncommutative \eqref{eq:classicalJMZ} with new transference methods, we shall enlarge the largest known subspace of the predual of the free group algebra for which the free Poisson semigroup $\lambda_\omega \mapsto \exp(-t|w|) \lambda_\omega$ converges bilaterally almost uniformly to initial data as $t \to 0^+$. 

\subsection*{{\bf A. The noncommutative Jessen-Marcinkiewicz-Zygmund theorem}}
\label{A}
Let us fix in what follows a quantum probability space $(\M,\tau)$. In order to generalize \eqref{eq:classicalJMZ} we start by introducing a notion of noncommutative $\limsup$ for finite von Neumann algebras. Let $c_0 \subset \ell_\infty$ be the subspace of sequences converging to $0$. Following the intuition that the $\limsup$ of a scalar-valued sequence $(a_n)_n$ satisfies the identity
\[
  \limsup_{n \to \infty} |a_n| = \| (a_n)_n \|_{\ell_\infty/c_0},
\] \vskip2pt
it would be natural to define the $L_p$-norm of the noncommutative $\limsup$ as the seminorm associated to $L_p[\ell_\infty]/L_p[c_0]$. Nevertheless, this definition does not work as intended, not even for sequences of functions in classical measure spaces as we explain in Remark \ref{rmk:WhyNotQuotient}. This imposes to work with a localized version of it. Given $(f_n)_n \subset L_p(\M)$ and recalling that we write $\Prj(\M)$ for the projection lattice in $\M$, set \vskip2pt
\begin{equation*}
  %\label{eq:DeflimsupIntro}
  \Big\| {\limsup_{n \to \infty}}^{+} f_n \Big\|_p
  \, = \, 
  \sup_{\varepsilon > 0} \, \inf_{\substack{e \in \Prj(\M)\\ \tau(e) < \varepsilon}} \, \inf_{N \ge 1} \,
  \Big\{ \big\| (e^\perp \, f_n \, e^\perp)_{n > N} \big\|_{L_p[\ell_\infty]} \Big\},
\end{equation*}
where the $\limsup^+$ is just a suggestive notation and does not correspond to a well-defined operator.
The mixed-norm space obtained by closure in that seminorm is denoted $L_p^\loc(\M;\ell_\infty/c_0)$, where the $\loc$ stems from the fact that we localize in corners by projections with trace arbitrarily close to $1$. In fact this space does arise as an actual quotient of $L_p(\M; \ell_\infty)$, further details will be given in the body of the paper. Like in the classical case, boundedness of the noncommutative $\limsup^+$ can be used as a tool to obtain bilateral almost uniform convergence, b.a.u. in short. But, contrary to the case of the supremum or the noncommutative $\sup^+$, a partial converse holds
\[
  \baulim_{n \to \infty} f_n = f \ \Longrightarrow \ \Big\| {\limsup_{n \to \infty}}^{+} f_n \Big\|_p \le \|f\|_p.
\]
Propositions \ref{prp:BAUImpLimsup} and \ref{prp:LimsupImpBAU} give further explanations. Here is our main technical result for the $\limsup^+$.

% Theorem + corollaries B y C
\begin{Atheorem}
\label{thm:MainTool}
Let $(\M,\tau)$ be a finite von Neumann algebra equipped with a tracial state $\tau$. Let us consider two families $(A_n)_n$ and $(B_m)_m$ of positivity-preserving operators in $L_1(\M)$. Assume the following conditions hold$\hskip1pt:$ 
  \begin{enumerate}[leftmargin=1.2cm, label={\rm (\roman*)}, ref={\rm (\roman*)}]
    \item \label{itm:MainTool.1}
    $\displaystyle{\big\| (B_m)_m: L_\Phi(\M) \to L_1(\M; \ell_\infty) \big\| < \infty}$.
    \item \label{itm:MainTool.2}
    There is an operator $F: L_1(\M) \to L_1(\M)$ such that $A_n(f) \to F(f)$ \textrm{b.a.u.} for $f \in L_1(\M)$.
  \end{enumerate}
  Then, the following inequality holds as well
  \begin{equation*}
    \Big\| {\limsup_{n,m \to \infty}}^{+} A_n \circ B_m (f) \Big\|_1
    \leq
    \| F \| \, \big\| (B_m)_m: L_\Phi(\M) \to L_1(\M; \ell_\infty) \big\| \, \| f \|_\Phi.
  \end{equation*}
\end{Atheorem}

Many noncommutative forms of \eqref{eq:classicalJMZ} arise easily from Theorem \ref{thm:MainTool}. Their proofs are elementary and follow after combining Theorem \ref{thm:MainTool} with noncommutative extrapolation and bilateral almost uniform convergence for elements of $L_1(\M)$. Let us illustrate what it gives for noncommutative martingales. We first give the most general formulation for the abstract composition of a couple of not necessarily commuting martingale filtrations. Then, we make explicit two cases of particular interest. In the first one we pick filtrations living in different components of a tensor product, which gives rise to the closest generalization of Jessen/Marcinkiewicz/Zygmund theorem. In the second one, we just replace tensor products by free products, which only makes sense in the noncommutative realm. As an illustration, the free product of two dyadic filtrations in the one-dimensional torus $\TT$ can only be defined in the free group algebra and the result below gives an estimate for the $L_1$-norm of the noncommutative $\limsup$ and almost uniform convergence in two directions with initial data in $L \log^2 L$.   

\begin{Acorollary}
  \label{cor:MultiMartingale}
  Let $(\M, \tau)$ be a finite von Neumann algebra and $(\M_n^{[i]})_n$ be nested
  sequences of von Neumann subalgebras
  $\M_n^{[i]} \subset \M_{n + 1}^{[i]} \subset \M$ for $i=1,2$.
  Let us denote the associated conditional
  expectations by $\EE_n^{[i]}: \M \to \M_n^{[i]}$. Then 
  \begin{enumerate}[leftmargin=1.2cm, ref={\rm (\roman*)}, label={\rm (\roman*)}]
    \item \label{itm:MultiMartingale.1}
    It holds that
    \[
      \Big\| \, {\limsup_{n, m \to \infty}}^{+} \big(\EE_n^{[1]} \circ \EE_m^{[2]}\big)(f) \Big\|_1
      \, \lesssim \,
      \| f \|_{L \log^2 L}.
    \]
    \item \label{itm:MultiMartingale.2}
    Let $(\M_n^{[i]})_n$ be a weak-$*$ dense filtration of $\M^{[i]} \subset \M$ for $i =1,2$. Assume in addition that $(\M, \tau)$  is either isomorphic to
    $(\M^{[1]} \weaktensor \M^{[2]},\tau_1 \otimes \tau_2)$ or to
    $(\M^{[1]} \ast \M^{[2]}, \tau_1 \ast \tau_2)$. Then, we respectively get
    \begin{eqnarray*}
    \Big\| \, {\limsup_{n, m \to \infty}}^{+} \big(\EE_n^{[1]} \otimes \EE_m^{[2]}\big)(f) \Big\|_1& \lesssim & \|f\|_{L \log^2 L}, \\ 
    \Big\| \, {\limsup_{n, m \to \infty}}^{+} \big(\EE_n^{[1]} \hskip1pt \ast \hskip2pt \EE_m^{[2]}\big)(f) \Big\|_1 & \lesssim & \|f\|_{L \log^2 L}.
    \end{eqnarray*}
    In both cases we get b.a.u. convergence to $f$ as $n, m \to \infty$,
    for every $f \in L \log^2 L(\M)$.
  \end{enumerate}
\end{Acorollary}

Although formulated in the biparametric case, Corollary \ref{cor:MultiMartingale} extends to $d$-parameters just by working with the Orlicz class $L \log^{2 (d - 1)} L$ instead. The additional logarithmic powers |compared to $L \log^{d-1}L$ in the classical form of \eqref{eq:classicalJMZ}| are expected from the sharp known results for noncommutative maximals and Yano's extrapolation recalled above. In fact, we believe this is best possible. Given Markovian maps $T_i: \M \to \M$, we recall that an analogous statement holds true replacing the above conditional expectations by (multiparametric) ergodic means $$M_n(T_i) = \frac{1}{n+1} \sum_{k=0}^n T_i^k$$ or subordinated Markovian semigroups, see Corollary \ref{cor:MultErgodic}. As another illustration, we may also consider free products of convolutions in $\TT$ against arbitrarily small balls $B_r(0)$, in the line of the original formulation of Jessen, Marcinkiewicz and Zygmund. This gives another form of \eqref{eq:classicalJMZ} in the free group algebra $\L \FF_n = L_\infty(\TT) \ast \cdots \ast L_\infty(\TT)$ with no analog for general von Neumann algebras.  

The main result in \cite{HongSun2018} establishes bilateral almost uniform convergence for certain $d$-parametric ergodic means. The given argument extends though to Markovian ergodic means and martingale approximations as well. The underlying \lq\lq maximal inequality" is however very different to ours since they prove an inequality of the form 
\begin{equation} \label{HongThm} \tag{HS} 
  \begin{split}
    \Big\| \hskip2pt {\sup_{n,m \ge 1}}^{\hskip-5pt +} & \underbrace{M_n(T_1) \circ M_m(T_2)(f)}_{f_{n,m}} \Big\|_{1,\infty} \\
  & :=
  \sup_{\lambda >. 0} \, \inf \Big\{ \lambda \, \tau \big( \1 - q(\lambda) \big) :  q(\lambda) f_{n,m} q(\lambda) \le \lambda \1 \Big\}
  \, \lesssim \,
  \|f\|_{L \log^2 L}
  \end{split}
\end{equation}
under the assumption that the ergodic means are pairwise commuting. The advantage of replacing the noncommutative $\sup$ by our noncommutative $\limsup$ in Corollary \ref{cor:MultiMartingale} is that we get an $L_1$ endpoint space, instead of weak $L_1$. This is particularly useful to use Fubini-type arguments, which will be crucial in our analysis of the free Poisson semigroup in the free group algebra below. Another advantage of our formulation is that we do not need our approximating maps to be pairwise commuting.  

\vskip-20pt

\null

% Conjecture + Our strong theorem + compare with Guixiang's 
\subsubsection*{{\bf B. The noncommutative C\'ordoba-Fefferman inequality}}
\label{B}
As we already explained for \eqref{eq:classicalJMZ} it is natural to expect that the noncommutative form \eqref{eq:classicalCF} should hold with data in $L \log^{2(d-1)} L$ in $d$-parameters. As usual, we fix $d=2$ for simplicity and assume that $\M$ is the tensor or free product of $\M^{[i]}$ ($i=1,2$), both equipped with a weak-$*$ dense martingale filtration with associated conditional expectations $\EE_n^{[i]}$. Let $f_{n,m} = \EE^{[1]}_n \otimes \EE_m^{[2]}(f)$ or $f_{n,m} = \EE^{[1]}_n \ast \EE_m^{[2]}(f)$ accordingly. Then we conjecture that the following holds for every $f \in L \log^{2} L(\M)$: 

\begin{Bconjecture}
  \label{cjn:StrongIneqC}
  Given $\lambda > 0$ and $f \geq 0$, there exists $q(\lambda) \in \Prj(\M)$ satisfying
  \begin{enumerate}[leftmargin=1.2cm, ref={\rm (\roman*)}, label={\rm (\roman*)}]
    \item \label{itm:StrongIneqC.1} $q(\lambda) f_{n,m} q(\lambda) \, \le \, \lambda \, q(\lambda)$.
    \item \label{itm:StrongIneqC.2}
    $\displaystyle{
      \tau \big( \1 - q(\lambda) \big)
      \, \lesssim \,
      \tau \bigg\{ \frac{f}{\lambda} \Big( 1 + \log_+ \Big( \frac{f}{\lambda} \Big) \Big)^2 \bigg\}.
    }$
  \end{enumerate}
  %$$q(\lambda) f_{n,m} q(\lambda) \, \le \, \lambda \, q(\lambda) \quad \mbox{and} \quad \tau \big( \1 - q(\lambda) \big) \, \lesssim \, \tau \bigg\{ \frac{f}{\lambda} \Big( 1 + \log_+ \Big( \frac{f}{\lambda} \Big) \Big)^2 \bigg\}.$$
\end{Bconjecture}

We also expect that the $2$ in the exponent of the logarithm is optimal. The conjectured inequality above can be understood as a noncommutative formulation of the weak Orlicz type $(\Phi,\Phi)$, as defined in \cite{Guzman1972ProductBases}, where $\Phi(t) = t (1 + \log_+ t)^2$. Indeed, after fixing $\Phi_1$, $\Phi_2:\RR_+ \to \RR_+$, it is said that an operator $f \mapsto f_n = A_n(f)$ is of \emph{weak type} $(\Phi_1, \Phi_2)$ iff for every $f \in L_{\Phi_1}(\M)_+$ and $\lambda > 0$, there is a $q(\lambda) \in \Prj(\M)$ satisfying $q(\lambda) f_{n,m} q(\lambda) \le \lambda q(\lambda)$ and that
\[
  \tau \big( q(\lambda)^\perp \big)
  \, \lesssim \, \tau \circ \Phi_2 \Big( \frac{f}{\lambda} \Big).
\]
Although Conjecture \ref{cjn:StrongIneqC} remains open, we have been able to obtain an
$\varepsilon$-perturbation

\begin{Btheorem}
  \label{thm:StrongIneq}
  Given $\lambda > 0$ and $f \geq 0$, there exists $q(\lambda) \in \Prj(\M)$ satisfying
  \begin{enumerate}[leftmargin=1.2cm, ref={\rm (\roman*)}, label={\rm (\roman*)}]
    \item \label{itm:StrongIneq.1} $q(\lambda) f_{n,m} q(\lambda) \, \lesssim_{(\varepsilon)} \, \lambda \, q(\lambda)$.
    \item \label{itm:StrongIneq.2}
      $\displaystyle{
        \tau \big( \1 - q(\lambda) \big)
        \, \lesssim_{(\varepsilon)} \,
        \tau \left\{ \frac{f}{\lambda} \Big( 1 + \log_+ \Big( \frac{f}{\lambda} \Big) \Big) \Big( 1 + \log_+\Big(\frac{f}{\lambda^{1-\varepsilon}}\Big) \Big)^{1+\varepsilon} \right\}.
      }$
  \end{enumerate}
  %$$q(\lambda) f_{n,m} q(\lambda) \, \le_{(\varepsilon)} \, \lambda \, q(\lambda) \quad \mbox{and} \quad \tau \big( \1 - q(\lambda) \big) \, \lesssim_{(\varepsilon)} \, \tau \left\{ \frac{f}{\lambda} \Big( 1 + \log_+ \Big( \frac{f}{\lambda} \Big) \Big) \Big( 1 + \log_+\Big(\frac{f}{\lambda^{1-\varepsilon}}\Big) \Big)^{1+\varepsilon} \right\}.$$
%  \tau \bigg\{ \frac{f}{\lambda} \Big( 1 + \log_+ \Big( \frac{f}{\lambda} \Big) \Big)^2 \bigg\}.$$
\end{Btheorem}

Theorem \ref{thm:StrongIneq} generalizes to any other filtrations satisfying $\EE_n^{[1]} \circ \EE_m^{[2]} = \EE_m^{[2]} \circ \EE_n^{[1]}$. Let us recall that Hong/Sun's inequality \eqref{HongThm} holds as well for commuting martingale filtrations. Their proof combines Cuculescu's theorem for $1$-parameter martingales with a well-known decomposition of $L \log^\alpha L$ into atoms. Let us note that \eqref{HongThm} gives 
\[
 \lambda \, \tau \big( \1 - q(\lambda) \big)
  \, \lesssim \,
 \| f \|_{L \log^2 L}.
\]
Their result seems a priori stronger, since it holds for every $f \in L \log^2 L(\M)$, while our result works only for $f$ in the smaller space $L \log^{2 + \varepsilon} L(\M)$. However, our inequality gives a faster rate of decay for $\lambda \mapsto \tau(\1 - q(\lambda))$. Indeed, the proof of \eqref{HongThm} requires going through $L_1[\ell_\infty]$, which forces to loose some factor with respect to the optimal size of the projections $q(\lambda)$. On the contrary, our proof uses a subtle refinement of Cuculescu's construction together with a smart choice for $q(\lambda)$ in terms of the two martingale filtrations. More details on the relation between both results will be given in Remark \ref{rmk:LongRemark}. To understand the importance of that difference, recall that Conjecture \ref{cjn:StrongIneqC} and Theorem \ref{thm:StrongIneq} give weak $(\Phi, \Phi)$-type inequalities, while \eqref{HongThm} is limited to a weak type $(\Phi,1)$ inequality which has mixed homogeneity. Having the correct weak Orlicz type may be of importance in order to apply real interpolation between Orlicz spaces. On the other hand, beyond \eqref{HongThm} we suspect that Theorem \ref{thm:StrongIneq} holds true for multiparametric ergodic means, but we have not investigated this topic. We also suspect that $L \log^2 L(\M)$ is optimal in terms of b.a.u. convergence and weak Orlicz type.

\begin{Bconjecture}
  \label{prb:OptimalOrlicz}
  Let $f_{n,m} = \EE_n^{[1]} \otimes \EE_m^{[2]}(f)$ be as above$\hskip1pt :$
  \begin{enumerate}[leftmargin=1.2cm, label={\bf \arabic*.}, ref={\bf \arabic*}]
    \item \label{itm:OptimalOrlicz.1}
    If $f \mapsto (f_{n,m})_{n,m}$ is of weak type $(\Phi,\Phi)$,
    then $\Phi(t) \geq c \, t (1 + \log_+ t)^2$
    for large $t$.
    \item \label{itm:OptimalOrlicz.2} 
    If $f_{n,m} \to f$ b.a.u. for every $f \in L_\Phi(\M)$,
    then $\Phi(t) \geq c \, t (1 + \log_+ t)^2$ for large $t$.
  \end{enumerate}
\end{Bconjecture}
The point \ref{itm:OptimalOrlicz.2} has a clear precedent in \cite[Theorem 8/Lemma F]{Jessen1935} in the Abelian case.

\vskip-20pt

\null

% Applications to the free group
\subsubsection*{{\bf C. Applications to the free group}}
\label{C}
One of our original motivations to study multiparametric bilateral almost uniform convergence comes from a problem in the free group algebra. Let $\FF_2 = \langle a, b \rangle$ be the free group on two generators and let $\omega \mapsto |\omega|$ be its standard word length. It is a very well known observation due to Haagerup \cite{Haa1979} that the function $\omega \mapsto e^{-t |\omega|}$ is positive definite
and therefore the \emph{free Poisson semigroup} $S_t: \L \FF_2 \to L \FF_2$ given by
\[
  S_t \Big( \sum_{\omega \in \FF_2} a_\omega \, \lambda_\omega \Big)
  = 
  \sum_{\omega \in \FF_2} e^{-t |\omega|} \, a_w \, \lambda_\omega
\]
is a symmetric Markovian semigroup. The following natural problem remains open:

\begin{Cproblem}
  \label{prb:ConvergenceFreeIntro}
  \
  \begin{enumerate}[leftmargin=1.2cm, label={\bf \arabic*.}, ref={\bf \arabic*}]
    \item \label{itm:ConvergenceFreeIntro.1}
    Does $S_t(f) \to f$ bilaterally almost uniformly for every $f \in L_1(\L \FF_2)$ as $t \to 0$\emph{?}
    \item \label{itm:ConvergenceFreeIntro.2}
    If not, what is the largest subclass of $L_1(\L \FF_2)$ over which $S_t(f) \to f$ b.a.u. as $t \to 0$\emph{?}
  \end{enumerate}
\end{Cproblem}

Currently, the largest known class is $L \log^2 L(\L \FF_2)$ and it was obtained in \cite{Hu2009}. It is far from clear that the answer to the first question above is positive. In fact, indirectly related work by Ornstein and Tao \cite{Ornstein1968Pointwise,Tao2015failure} indicates that b.a.u. convergence might fail in the whole $L_1(\L \FF_2)$. Our methods in this paper will provide a significantly larger subspace of $L_1(\L \FF_2)$ where bilateral almost uniform convergence holds. Recall that, after identifying $\L \FF_2$ with $\L (\ZZ \ast \ZZ) \cong L_\infty(\TT) \ast L_\infty(\TT)$, the free Poisson semigroup factorizes as the free product $S_t = P_t \ast P_t$ where $P_t$ is the classical Poisson semigroup in $\TT$. Theorem \ref{thm:MainTool} easily gives 
\[
  \baulim_{s, t \to 0^+} \, (P_s \ast P_t)f = f \quad \mbox{ for every } f \in L \log^2 L(\L \FF_2).
\]
It is tempting to use that both free components of $\L \FF_2$ are Abelian |so that there is almost uniform convergence for $f \in L \log L(\TT)$| to bootstrap the result above, proving b.a.u. convergence for every $f \in L \log L(\L \FF_2)$. This does not work as intended, but we still obtain an improvement. Let $\Sigma \subset \FF_2$ be the subset of reduced words $s_1^{n_1} \, s_2^{n_2} \, ... \, s_r^{n_r}$ with no sign changes in the exponents of the generators $a$ and $b$. The usefulness of $\Sigma$ is that the group homomorphism $\FF_2 \to \ZZ \times \ZZ$ given by removing the $a$ and $b$ respectively \emph{detects} the length $|\omega|$ if $\omega \in \Sigma$. This facilitates the application of transference techniques for elements $f \in \L \FF_2$ supported in $\Sigma$. Let us define 
\[
  L_1(\L \FF_2){|}_{\Sigma} = \Big\{ f \in L_1(\L \FF_2) : \tau( f \, \lambda_\omega^\ast ) = 0, \forall \omega \not\in \Sigma \Big\}.
\]
\begin{Ctheorem}
  \label{thm:FreegroupExt}
  Consider the space
  \[
    \C = L_1(\L \FF_2){|}_{\Sigma} + L \log^2 L(\L \FF_2),
  \]
  \[
    \| f \|_\C =
    \inf \Big\{ \| g \|_1 + \| h \|_{L \log^2 L}
                : f = g + h \mbox{ and } g \in L_1(\L \FF_2){|}_\Sigma \Big\}.
  \]
  Then, if $t_n \to 0^+$, we have that $(S_{t_n})_{n}$ satisfies
  \begin{equation}
    \label{eq:FreegroupExt}
    \Big\| {\limsup_{n \to \infty}}^+ S_{t_n}(f) \Big\|_1 \, \lesssim \, \| f \|_\C.
  \end{equation}
  In particular, $S_t(f) \to f$ bilaterally almost uniformly as $t \to 0^+$ for every element $f \in \C$. 
\end{Ctheorem}

Theorem \ref{thm:FreegroupExt} is related to a few recent results in the literature \cite{JunMeiPar2014Riesz, MeiRicard2017FreeHilb, MeiXu2019Free} which have established $L_p$-boundedness of Fourier multipliers over special subsets of $\FF_2$. In \cite{JunMeiPar2014Riesz} it was proved that H\"ormander-Mikhlin multiplier theorem and Littlewood-Paley square function inequalities admit a generalization when we restrict to arbitrary branches of the free group algebra. Next, these square function estimates were greatly improved in \cite{MeiRicard2017FreeHilb} after proving the $L_p$-boundedness of the truncation into branches with new Hilbert transform methods. Finally, in \cite{MeiXu2019Free} the authors proved that certain specific multipliers are $L_p$-bounded iff the restriction of the Fourier multiplier to the torus $\TT$ |the branch containing the integer powers of one generator| is $L_p$-bounded. The price for the equivalence with the classical theory is that the admissible multipliers are constant in large regions of $\FF_2$ and does not include natural operators like the free Poisson semigroup. In fact,  these multipliers are determined by its value in a linearly growing subset of $\FF_2$. In conclusion, in these previous results, the reference subsets over which the theory was built were relatively simple in terms of their geometry, like branches or other linear growth subsets. In that respect, it is worth noting that our reference set $\Sigma$ grows exponentially with respect to word length
\[
  \big| \Sigma \cap \big\{ \omega \in \FF_2 : |\omega| \leq n \big\} \big| \sim 2^n.
\]

\section{\bf Preliminaries \label{sct:preliminaries}}

In this section we will briefly review basic definitions and results that will be
used throughout the article. 

% von neumann algebras and traces, finite ans semifinite
% measurable operators and lp spaces, orlicz spaces
\subsection*{Noncommutative integration}
We denote by $\B(H)$ the $\ast$-algebra of bounded operators on a Hilbert space $H$. A \emph{von Neumann algebra} is a $\ast$-subalgebra $\M \subset \B(H)$ that is closed in the weak-$\ast$ topology of $\B(H)$. Naturally, von Neumann algebras are dual algebras that come equipped with a weak-$\ast$ topology inherited from that of $\B(H)$. The interested reader can look more on the basic theory of von Neumann algebras in the many texts available, such as \cite{TaI, TaII}. 
A a contractive and positivity-preserving linear map in $\M^\ast$ is called a state. We will say that it is \emph{normal} when it is weak-$\ast$ continuous. A \emph{finite trace} or simply a \emph{trace} will be a normal state $\tau: \M \to \CC$ satisfying $\tau(f g) = \tau(g f)$ for every $f, g \in \M$. We will say that such a state is \emph{faithful} whenever $\tau(f) = 0$ implies that $f = 0$ for $f \in \M_+ = \{ g \in \M: g \geq 0\}$. A von Neumann algebra admitting a faithful finite state is said to be \emph{finite}. Throughout this text we will work in the context of finite von Neumann algebras. 

Many measure-theoretical notions extend naturally to the context of finite von Neumann algebras. We will statt with the construction of noncommutative $L_p$-spaces and Orlicz spaces, which can be consulted in detail in \cite{PiXu2003, Terp1981lp, Kunze1990Orlicz}. Recall that the $p$-norm is given by
\[
  \| f \|_p = \tau(|f|^p)^\frac1{p}
\]
We can define $L_p(\M,\tau)$ as the metric closure of $\M$ with respect to the metric induced by the $p$-norm above. There is an alternative characterization of the noncommutative $L_p$-spaces. For that it is needed to construct the space of $\tau$-measurable unbounded operators $L_0(\M,\tau)$, a locally convex space analogous to the classical space of measurable operators with the topology of convergence in measure. Then, the $L_p$-spaces can be described as the subset of the $\tau$-measurable operators satisfying that $\tau(|f|^p) < \infty$, see \cite{Terp1981lp}.
%Let us start recalling that an unbounded operator $f: \dom(f) \subset H \to H$ is said to be affiliated with $\M$ iff it commutes with every unitary $u$ in $\M'$, the commutant of $\M$. A densely defined and closable operator $f$ affiliated with $\M$ is $\tau$-measurable iff for every $\delta > 0$ there is an orthogonal projection $p \in \Prj(\M)$ such that $\tau(p^\perp) < \delta$ and $p H \subset \dom(f)$. We will denote the space of all $\tau$-measurable operators by $L_0(\M,\tau)$. $L_0(\M,\tau)$ can be endowed with a locally convex topology, analogous to the topology of convergence in measure in the classical case, that turns $L_0(\M,\tau)$ into a topological $\ast$-algebra, see \cite{Terp1981lp} for the details. The trace $\tau$ can be extended to the whole $L_0(\M,\tau)_+$. Therefore, given any exponent $1 \leq p < \infty$, we can define the \emph{noncommutative $L_p$-space} associated with $(\M,\tau)$ by
%\[
%  L_p(\M, \tau) = \big\{ f \in L_0(\M,\tau) : \tau(|f|^p) < \infty \big\}.
%\]
%$L_p$ spaces are complete with respect to the norm given by $\| f \|_p = \tau(|f|^p)^\frac1{p}$ and thus, they can alternatively be seen as metric closure of $\M$ with respect to that norm. We will use the convention $L_\infty(\M) = \M$ and denote the operator norm of $\M$ by $\| \cdot \|_\infty$.
In the cases $p = 1$ and $p = 2$, $L_p(\M, \tau)$ can be easily characterized as the predual of $\M$ ---denoted by $\M_\ast$--- and the GNS construction associated to $\tau$, respectively. 

We will also work throughout this text with \emph{noncommutative Orlicz spaces}. Let $\Phi: \RR_+ \to \RR_+$ be a convex increasing function with $\Phi(0) = 0$ and $\Phi(\infty) = \lim_{t \to \infty} \Phi(t) = \infty$. Given such function, we define $L_\Phi(\M,\tau)$ as
\[
  L_\Phi(\M, \tau) = \Big\{ f \in L_0(\M,\tau) : \exists \lambda \geq 0, \,  \tau \circ \Phi \Big( \frac{|f|}{\lambda} \Big) < \infty \Big\}.
\]
Those are Banach spaces when endowed with the norm given by
\[
  \| f \|_{L_\Phi} = \inf \Big\{ \lambda \geq 0 : \tau \circ \Phi \Big( \frac{|f|}{\lambda} \Big)  \leq 1 \Big\}.
\]
We will omit the dependency on $\tau$ whenever it is clear from the context. Among all Orlicz spaces we will use mainly those where $\Phi$ is of the form $\Phi(t) = t \, (1 + \log_{+} t)^\alpha$, where $\log_{+}(t) = \max \{0, \log(t) \}$. In that case we will denote the space by $L \log^\alpha L(\M,\tau)$.

One of the most common examples of finite von Neumann algebras is given by $\B(\ell_2^m)$ with its unique trace $\Tr$. In that case, the associated $L_p$-spaces are known as \emph{Schatten classes} $S_p^m$. 

Throughout this text we will make brief use of the theory of operator spaces, specially around the discussion before Problem \ref{prb:CBMaximal}. The category of operator spaces is that which is obtained when one takes the categor y of Banach spaces and replaces bounded maps by \emph{completely bounded maps} i.e. maps whose operator norms are stable by matrix amplifications. The interested reader can look in some of the already available texts \cite{Pi2003, EffRu2000Book} as well as the beginning of \cite{BleMer2004Operator}. One key observation about operator spaces that we will use is that the complete norm of a map can be detected by operator space valued Schatten classes $S_p[E]$, see \cite{Pi1998} for the construction. Indeed, let $\phi: E \to F$ be a completely bounded map, for every $1 \leq p \leq \infty$, we have that
\begin{equation}
  \label{eq:CBSchattenNorm}
  \| \phi: E \to F \|_\cb = \sup_{m \geq 1} \Big\{ \big\| \Id \otimes \phi: S_p^m[E] \to S_p^m[E] \big\| \Big\}.
\end{equation}
One consequence of the identity above is that the operator space structure of an operator space $E$ is completely defined by the family of norms of $S_p^m[E]$. We will use this often since for $L_p(\M)$, we have that $S_p^m[L_p(\M)] = L_p(M_m(\CC) \otimes \M)$. Other spaces related to $L_p$ admit an equally explicit description. 

\subsection*{Noncommutative maximal inequalities}
The theory of maximal inequalities over noncommutative $L_p$-spaces is not without subtleties. The main one being that, given a family of positive non commuting operators $(f_n)_n$, there is no natural way of defining its supremum in a way that commutes with taking vector states, see \cite[p. 386]{JunXu2007}. Nevertheless, it is possible to define a quantity that behaves like the $L_p$ norm of the supremum of such a family by working with the mixed norm spaces $L_p(\M;\ell_\infty)$. In \cite{Pi1998}, Pisier introduced a way of defining spaces of the form $L_p(\M;E)$ for hyperfinite $\M$, generalizing the Banach-valued $L_p$ spaces over a measure space. His method involves taking a direct limit of the $E$-valued Schatten classes $S_p^m[E]$ introduced in the previous paragraph. Later, this notion was extended to the setting of QWEP von Neumann algebras by Junge \cite{Junge2004Fubini}. In the case of $E = \ell_\infty$ an ad hoc definition requiring a factorization was introduced in \cite{Jun2002Doob}, see also \cite{JunXu2007}. In this case one does not need any approximation properties of the von Neumann algebra. This is the case that we will review here. Fix $1 \leq p  \leq \infty$ and a semifinite von Neumann algebra $\M$ with a n.s.f. trace $\tau$.

\begin{definition}
  \label{def:MixedL1Linfty}
Given a bounded sequence $(f_n)_n \subset L_p(\M)$ it is said that $(f_n)_n \in L_p(M;\ell_\infty)$ iff there exists a factorization $f_n = \alpha \, g_n \, \beta$ such that $\alpha, \beta \in L_{2p}(\M)$ and $g_n \in \ell_\infty \weaktensor \M = \ell_\infty[\M]$. The norm in $L_p(\M;\ell_\infty)$ is given by
  \begin{equation}
    \big\| (f_n)_n \big\|_{L_p[\ell_\infty]}
      := \inf \Big\{ \| \alpha \|_{2 p} \, \sup_{n} \| h_n \|_\infty \, \| \beta \|_{2 p} \Big\},
  \end{equation}
  where the infimum is taken over all possible decompositions $f_n = \alpha \, h_n \, \beta$. 
\end{definition}

We define $L_p(\M;c_0)$ as the norm-closure of the finitely supported sequences $(f_n)_{0 \leq n \leq N}$ inside $L_p(\M;\ell_\infty)$. As a shorthand notation we will write 
\[
  L_p(\M;\ell_\infty) = L_{2 p}(\M) \, (\ell_\infty \weaktensor \M) \, L_{2 p}(\M).
\]
The same convention will be employed whenever a norm is given by the infimum over products. Note that it is easy to replace $\ell_\infty$ by a general von Neumann algebra $\N$ in the above expression to get $L_p(\M;\N) = L_{2p}(\M) \, (\M \weaktensor \N) \, L_{2p}(\M)$. One interesting case is $\N = L_\infty(\RR_+)$. This is the space that we will use to formulate maximal inequalities with respect to a continuous parameter. For purely conventional reasons we will often denote 
\[
  \Big\| \, {\sup_{n}}^{+} f_n \Big\|_p := \| (f_n)_n \|_{L_p[\ell_\infty]}.
\]
Here $\sup^{+}$ is just an evocative notation, since we already said that the noncommutative supremum may not be well defined. We will frequently use the following well known characterization of the $L_p[\ell_\infty]$-norm for $f_n \geq 0$:
\begin{equation}
  \label{eq:L1Linfty}
  \| (f_n)_n \|_{L_p[\ell_\infty]} = \inf_{f_n \leq g} \| g \|_p.
\end{equation}

It is possible to characterize the natural operator space structure of $L_p(\M;\ell_\infty)$ as the only one satisfying that 
\begin{equation}
  \label{eq:SchattenMaximal}
  S_p^m \big[ L_p(\M;\ell_\infty) \big]
  = L_p(M_m \otimes \M; \ell_\infty),
\end{equation}
see \cite{JunXu2007}.
Another property that will be useful is the fact that $L_p(\M;\ell_\infty)$ is a bimodule with respect to $\M$, i.e.:
\begin{equation}
  \label{eq:bimoduleProp}
  \big\| a \cdot (f_n)_n \cdot b \big\|_{L_p[\ell_\infty]}
  \leq 
  \| a \|_{\M} \, \big\| (f_n) \big\|_{L_p[\ell_\infty]} \, \| b \|_\M.
\end{equation}
It also holds that $L_p(\M;\ell_\infty)$ is a bimodule with respect to the central action of $\ell_\infty$ acting as pointwise multiplicators, i.e:
$\| (\eta_n \cdot f_n)_n \big\|_{L_p[\ell_\infty]} \leq \| (\eta_n)_n \|_{\ell_\infty} \, \| (f_n) \|_{L_p[\ell_\infty]}$. Taking $\eta = \1_{[0,N]}$
gives the following monotonicity property with respect to the index set, which will be useful in the next section
\begin{equation}
  \label{eq:monotonicLpLinfty}
  \big\| (f_n)_{N > n} \big\|_{L_p[\ell_\infty]} \leq \big\| (f_n)_n \big\|_{L_p[\ell_\infty]}.
\end{equation}

One may also define $\ell_\infty$-valued Orlicz spaces $L_\Phi(\M;\ell_\infty)$. Following the notation of \cite[Definition 3.2]{Bekjan2017}, we will use the following notation
\[
  \tau \circ \Phi \Big( {\sup_{n}}^+ f_n \Big)
  = \inf \bigg\{ \frac{\tau \circ \Phi(|\alpha|^2) + \tau \circ \Phi(|\beta|^2)}{2} \, \Big( \sup_n \| g_n \|_\infty \Big)
                : f_n = \alpha \, g_n \, \beta
         \bigg\}.
\]
Then, the norm of $L_\Phi[\ell_\infty]$ is given by
\[
  \big\| (f_n)_n \big\|_{L_\Phi[\ell_\infty]}
  =
  \inf \left\{ \lambda \geq 0 :
         \tau \circ \Phi \Big( \frac1{\lambda} \, {\sup_{n}}^+ f_n \Big) \leq 1 
       \right\}.
\]
In this context the identity \eqref{eq:L1Linfty} holds up to constants: if $f_n \geq 0$ then we have the following, see \cite{Bekjan2017}
\begin{equation}
  \label{eq:AltnormOrlicz}
  \big\| (f_n)_n \big\|_{L_\Phi[\ell_\infty]}
  \, \sim \, \inf \big\{ \| g \|_{L_\Phi} : f_n \leq g \big\},
\end{equation}

% almost uniform convergence
\subsection*{Almost uniform convergence\label{ssc:bau}}
Next we recall the definitions of almost uniform convergence in semifinite von Neumann algebras. The following one goes back at least to \cite{Lance1976Ergodic}, see also \cite{Cuculescu1971} as well as \cite{Yeadon1977, Yeadon1980II, Jajte1985Prob, Jajte1991L2}. 
\begin{definition}
  \label{def:bauConvergence}
 Given $(f_n)_n \subset L_0(\M,\tau)$, we will say that $f_n \to f$ \emph{bilaterally almost uniformly},
  or \emph{b.a.u.} in short, iff for every $\epsilon > 0$, there is a projection $e \in \Prj(\M)$ such that
  \begin{enumerate}[leftmargin=1.2cm, label={\rm (\roman*)}, ref={\rm (\roman*)}]
    \item \label{itm:bauConvergence.1}
    $\tau(e) < \epsilon$.
    \item \label{itm:bauConvergence.2}
    $\displaystyle{ \big\| e^\perp \, (f_n - f) \, e^\perp \big\|_\infty \longrightarrow 0}$ as $n \to \infty$.
  \end{enumerate}
\end{definition}
The use of the term ``bilateral'' in the definition above refers to the fact that the projection $e^\perp$ acts both on the right and on the left. One can similarly define the following stronger notions of almost uniform convergence
\[
  \begin{array}{lcl}
    f_n \to f \mbox{ column almost uniformly }
    & \Longleftrightarrow & \displaystyle{ \big\| (f_n - f) \, e^\perp \big\|_\infty \to 0}\\
    f_n \to f \mbox{ row almost uniformly } 
    & \Longleftrightarrow & \displaystyle{ \big\| e^\perp \, (f_n - f) \big\|_\infty \to 0}
  \end{array}
\]
It is also useful to work with combinations of these modes of convergence. For example, we say that $f_n \to f$ converges column$+$row almost uniformly iff there is a decomposition $f_n - f = a_n + b_n$ such that $a_n$ and $b_n$ converge column and row almost uniformly respectively, see \cite{Defant2011}. The terminology or \emph{row} and \emph{column} comes from the fact that
\[
  \begin{array}{lcl}
    f_n \to f \mbox{ column almost uniformly}
    & \quad \Longleftrightarrow \quad & (f_n - f)^\ast (f_n - f) \to 0 \mbox{ bilaterally almost uniformly}\\
    f_n \to f \mbox{ row almost uniformly}
    & \quad \Longleftrightarrow \quad & (f_n - f) (f_n - f)^\ast \to 0 \mbox{ bilaterally almost uniformly}\\
  \end{array}
\]
and the analogous terminology used in the theory of operator spaces, see \cite[Section 3.4]{EffRu2000Book}. 

The above definitions generalize the notion of almost everywhere convergence in the case of finite measure spaces. Indeed, by Egorov's theorem (see \cite[2.33]{Foll1999}) a sequence of measurable functions $f_n$ converges almost everywhere to $f$ iff for every $\epsilon > 0$ there is a measurable set $E$ with $\mu(E) < \epsilon$ and such that $\1_{E^C} (f_n - f)$ converge uniformly to $0$.

Like in the classical case, the standard approach to prove that a family of operators $S_n(f)$ converge bilaterally almost uniformly for every $f$ in a $L_p$ or Orlicz space is to try to obtain a noncommutative maximal inequality for elements $(S_n(f))_{n}$ i.e.
\[
  \big\| (S_n)_n: L_p(\M) \to L_p(\M;\ell_\infty) \big\| < \infty.
\]
We will use a modification of this principle to prove Theorem \ref{thm:MainTool}. Similarly, when the goal is to prove column/row almost uniform convergence asymmetric versions of the space $L_p(\M;\ell_\infty)$ given by
\begin{eqnarray*}
  L_p(\M;\ell_\infty^r) & = & L_p(\M) \, \big( \ell_\infty \weaktensor \M \big)\\
  L_p(\M;\ell_\infty^c) & = & \big( \ell_\infty \weaktensor \M \big) \, L_p(\M)
\end{eqnarray*}
have to be used, see \cite{Hong2016asymmetric} for the details.

% Extrapolation and interpolation
\subsection*{Extrapolation and interpolation for maximal inequalities}
Two tools that we are going to use repeatedly are the noncommutative analogues of Marcinkiewicz interpolation and Yano's extrapolation \cite{Yano1951} in the case of maximal operators. The particular noncommutative analogue of Marcienkiewicz interpolation that we are going to use was obtained in \cite{JunXu2007} while the noncommutative extrapolation was obtained in \cite{Hu2009}. Let us start by recalling the following definition, which extends the classical notion of weak type $(p,p)$ to von Neumann algebras. 

\begin{definition}
  \label{def:Weaktype}
  Let $(S_n)_n$ be a family of positivity preserving operators $S_n: L_p(\M) \to L_0(\M)$. $(S_n)_n$ is said to be of \emph{maximal weak type $(p,p)$} iff for every $f \in L_p(\M)_+$ and $\lambda > 0$ there is a projection $e \in \Prj(\M)$ such that
  \begin{enumerate}[leftmargin=1.2cm, label={\rm (\roman*)}, ref={\rm (\roman*)}]
    \item $\displaystyle{\tau(e^\perp) \lesssim \frac{\| f \|_p^p}{\lambda^p}}$
    \item $\displaystyle{e \, S_n(f) \, e \leq \lambda \, e}$ for every $n$.
  \end{enumerate}
\end{definition}

Observe that every family that induces a bounded map $(S_n)_n: L_p(\M) \to L_p(\M;\ell_\infty)$ is of maximal weak type $(p,p)$. In the case of Abelian von Neumann algebras, this definition coincides with the associated maximal operator $\R(f) = \sup_n |S_n(f)|$ being of weak type $(p,p)$.

\begin{theorem}[{\bf \cite[Theorem 3.1]{JunXu2007}}]
  \label{thm:JunXuRealInterpolation}
  Assume $(S_n)_n$ is a family of positivity preserving operators $S_n: L_{p_0}(\M) + L_{p_1}(\M) \to L_0(\M)$
  that are of maximal weak type $(p_0,p_0)$ and $(p_1, p_1)$ for $1 \leq p_0 < p_1 \leq \infty$.
  Then, for every $p_0 < p < p_1$
  \[
    \big\| (S_n)_n: L_p(\M) \to L_p(\M;\ell_\infty) \big\|
    \lesssim
    \max \bigg\{ \Big(\frac1{p} - \frac1{p_1} \Big)^{-2}, \Big(\frac1{p_0} - \frac1{p} \Big)^{-2} \bigg\}.
  \]
\end{theorem}

It is very interesting to note that the constant appearing in the exponent depending of $p$, $p_0$ and $p_1$ is $-2$ instead of $-1$ like in the classical case. The order of magnitude of the constant was also shown to be optimal in \cite{JunXu2007}, proving that the commutative and noncommutative theories diverge significantly. 

We will use the result above in the cases in which $(S_n)_n$ is simultaneously of weak type $(1,1)$ and bounded as a map $(S_n)_n: \M \to \ell_\infty \weaktensor \M$. In that case the order of the norm of $(S_n)_n: L_p(\M) \to L_p(\M;\ell_\infty)$ in terms of $p$ is given by $(p-1)^{-2}$. It turns out that extrapolation results give a sort of reciprocal near $p=1$. Indeed, the following holds:

\begin{theorem}[{\bf \cite[Theorem 2.3]{Hu2009}}]
  \label{thm:Extrapolation}
  Let $(\M,\tau)$ be a finite von Neumann algebra with a faithful tracial state
  and $(S_n)_n: L_1(\M) \cap \M \to L_1(\M) + \M$ a family of positivity preserving
  operators satisfying that
  \[
    \big\| (S_n)_n: L_p(\M) \to L_p(\M;\ell_\infty) \big\|
    \, \lesssim \,
    \max \Big\{ 1, \Big( \frac1{p - 1} \Big)^\alpha \Big\},
  \]
  for every $1 < p \leq \infty$. Then
  \[
    \big\| (S_n)_n: L \log^\alpha L(\M) \to L_1(\M;\ell_\infty) \big\|
    \, < \infty
  \]
\end{theorem}

% Ergodic means 
\subsection*{Ergodic means\label{ssn:ErgodicMeans}} Let $(\M,\tau)$ be a finite von Neumann algebra. We consider the following properties for a linear operator $T:\M \to \M$:
  \begin{enumerate}[leftmargin=1.2cm, label={\rm (\Roman*)}, ref={\rm (\Roman*)}]
    \item \label{itm:MarkovianOpI} $T:\M \to \M$ is a normal contraction.
    \item \label{itm:MarkovianOpII} $T$ is \emph{completely positive}.
    \item \label{itm:MarkovianOpIII} $T$ is \emph{$\tau$-preserving}, i.e. $\tau \circ T = \tau$.
    \item \label{itm:MarkovianOpIV} $T$ is \emph{symmetric}, i.e.
    $\tau(f^\ast \, T(g)) = \tau(T(f)^\ast \, g)$ for every $f, g \in \M$.
  \end{enumerate}
  
If $T$ satisfies properties \ref{itm:MarkovianOpI}-\ref{itm:MarkovianOpIII} we will say that it is a \emph{Markovian} operator. If it satisifies \ref{itm:MarkovianOpIV} as well, then we say that it is a \emph{symmetric Marovian} operator. If $T$ satisfies that $\tau \circ T \leq \tau$ instead if \ref{itm:MarkovianOpIII}, then it is said to be \emph{submarkovian}. It is an immediate consequence of the Riesz interpolation that symmetric Markovian operators extend as contractions to $L_p$ spaces, $1 \leq p \leq \infty$. One can see that Markovian operators are also unital. Given an operator $T$ we define its \emph{ergodic means} $M_n(T)$ by
\begin{equation}
  \label{{eq:IntroErgodicMean}}
  M_n(T) = \frac1{n + 1 } \sum_{k = 0}^n T^k.
\end{equation}
Observe that $M_n(T)$ is Markovian when $T$ is and symmetric when $T$ is. Consider the subset of fixed points by $T$ in $\M$
\[
  \ker_\infty(T - \1) =  \{ f \in \M : T(f) = f \}
\]
Clearly $\ker_\infty(T - \1)$ is a weak-$\ast$ closed operator system. It is well known, see  \cite{JunXu2007} and references therein, that it is complemented and induces a splitting
\begin{equation}
  \label{eq:DecFixedPoint}
  \M = \ker_\infty(T - \1) \oplus \overline{\mathrm{Im}[T - \1]^{\wast}},
\end{equation}
where the complement is the weak-$\ast$ closure of the image of $(\1 - T)$. Denote by $F$ the projection from $\M$ to $\ker_\infty(T - \1)$. An important observation is that the decomposition in \eqref{eq:DecFixedPoint} extends to more general spaces associated to $\M$, like $L_p$ or Orlicz spaces. We will abuse notation and denote by $F$ the projection $L_p(\M) \to \ker_p(T - \1)$ as well. The following proposition gathers a few known results.

\begin{proposition}
  \label{prp:ErgodicNorm}
  Let $T:\M \to \M$ be a Markovian operator.
  \begin{enumerate}[leftmargin=1.2cm, label={\rm (\roman*)}, ref={\rm (\roman*)}]
    \item \label{itm:ErgodicNorm.1} For every $f \in L_1(\M)$, $\| M_n(T)(f) - F(f) \|_1 \to 0$.
    \item \label{itm:ErgodicNorm.2} For every $f \in \M$, $M_n(T)(f) \to F(f)$ in the weak-$\ast$ topology.
    \item \label{itm:ErgodicNorm.3} There is a weak-$\ast$ dense class $\S \subset \M$ which is norm dense in $L_p(\M,\tau)$ for $1 \leq p < \infty$ and in $L \log^\alpha L(\M)$, such that $\| M_n(T)(f) - F(f) \|_\infty \to 0$ for every $f \in \S$.
  \end{enumerate}
\end{proposition}

Let us justify \ref{itm:ErgodicNorm.3}, we can use the decomposition in \eqref{eq:DecFixedPoint} and the fact that the complement of $\ker_\infty(T - \1)$ is the weak-$\ast$ closure of $(\1 - T)[\M]$, to define $S = \ker_\infty(T - \1) + (\1  - T)[\M]$. Now, if $g \in \S$ is of the form $g = g_1 + (g_2 - T(g_2))$ with $g_1 \in \F_\infty(T)$ and $g_2 \in \M$, we have
\[
  M_n(T)(g) = g_1 + \frac{1}{n + 1} (g_2 - T^{n + 1}(g_2))
\] 
and therefore 
\[
  \| F(g) - M_n(T)(g) \|_\infty \leq 2 \frac{\| g_2 \|_\infty}{n + 1},
\] 
which shows the convergence in $\M$. The norm density result for spaces of the form $L \log^\alpha L(\M)$ and for $L_p(\M)$ with $1 \leq p < \infty$ follows similarly.

Based on earlier results for the finite case that appeared in \cite{Lance1976Ergodic}, it was proved in \cite{Yeadon1977, Yeadon1980II} that the ergodic means $(M_m(T))_m$ are maximally of weak type $(1,1)$. Since they are trivially bounded as a map $(S_n)_n: \M \to \ell_\infty \weaktensor \M$ we can apply Theorem \ref{thm:JunXuRealInterpolation} to get that $(M_m(T))_m$ is bounded as a map $(M_m(T))_m: L_p(\M) \to L_p(\M; \ell_\infty)$. The combination of this result with the noncommutative extrapolation in Theorem \ref{thm:Extrapolation} gives the following result, which was obtained in  \cite{JunXu2007} and \cite{Hu2009}:
\begin{theorem}
  \label{thm:LpBoundErgodic}
  Let $(\M,\tau)$ be a finite von Neumann algebra with a faithful tracial state and $T$ a Markovian operator.
  \begin{enumerate}[leftmargin=1.2cm, label={\rm (\roman*)}, ref={\rm (\roman*)}]
    \item \label{itm:LpBoundErgodic.1}
    For every $1 < p \leq \infty$, 
    \[
      \big\| (M_n(T))_n: L_p(\M) \to L_p(\M;\ell_\infty) \big\|
      \, \lesssim \,
      \max \Big\{ 1, \Big( \frac1{p - 1} \Big)^2 \Big\}.
    \]
    \item \label{itm:LpBoundErgodic.2}
    $\displaystyle{\big\| (M_n(T))_n: L \log^2 L(\M) \to L_1(\M;\ell_\infty) \big\| < \infty}$
    \item \label{itm:LpBoundErgodic.3}
    For every $f \in L_1(\M)$, $M_m(T)(f) \to F(f)$ bilaterally almost uniformly as $m \to \infty$.
    \item \label{itm:LpBoundErgodic.4}
    If $T$ is symmetric Markovian, then \ref{itm:LpBoundErgodic.1} and \ref{itm:LpBoundErgodic.2} hold with $T^m$ instead of $M_m(T)$.
  \end{enumerate}
\end{theorem}
  
It is worth noting that the analogue of \ref{itm:LpBoundErgodic.3} is false when $M_m(T)$ is replaced by $T^m$ even if $T$ is symmetric. The reason is that $T^m$ may not satisfy a maximal weak $(1,1)$ inequality and the passage from $M_m(T)$ to $T^m$ requires and argument based on interpolation due originally to Stein (see \cite{Ste1970}). Moreover, there are Abelian examples of Markovian operators such that $T^m(f)$ diverges almost everywhere for some $f \in L_1$, see \cite{Ornstein1968Pointwise}. Nevertheless, it holds that $T^m(f) \to f$ for every $f \in L \log^2 L(\M)$, when $\T$ is symmetric. 

% Semigroups
\subsection*{Semigroups} Recall that a semigroup of operators is a family $(S_t)_{t \geq 0}$ that satisfies $S_t \circ S_s = S_{t + s}$ and $S_0 = \Id$.

\begin{definition}
  \label{def:MarkovianSemigroup}
  Let $(\M,\tau)$ be a finite von Neumann algebra. A semigroup $(S_t)_{t \geq 0}$ is called Markovian iff
  each $S_t:\M \to \M$ is symmetric Markovian and the map $t \mapsto S_t$ is pointwise weak-$\ast$ continuous.
\end{definition}

The theory of Markovian semigroups and their ergodic theory can be understood as a continuous analogue of
that of Markovian operators. With that in mind, given $s \geq 0$ we define its ergodic mean associated to the semigroup
$\S=(S_t)_{t \geq 0}$ by
\[
  M_s(\S)(f) = \frac1{s} \int_0^s S_t(f) \, d  t = \dashint_0^s S_t(f) \, d  t.
\]
The Dunford-Schwartz ergodic theorem can be extended to the noncommutative setting: the family $(M_s(\S))_{s \geq 0}$ is of maximal weak type $(1,1)$, see \cite{JunXu2007}. As a consequence of that, we have that there is bilateral almost uniform convergence of $M_s(\S)f$ to $f$ as $s \to 0^+$. In contrast, the problem of determining whether $S_t(f) \to f$ bilaterally almost uniformly for $f \in L_1(\M)$ is far more difficult, since the operators $(S_t)$ may not satisfy a maximal inequality. Nevertheless, there is a class of semigroups which have bilateral almost uniform convergence for every element in $L_1(\M)$, the so called \emph{subordinated semigroups}. Given semigroups $\S = (S_t)_{t \geq 0}$ and $\T = (T_t)_{t \geq 0}$, $\S$ is subordinated to $\T$ iff there is a family of functions $\phi_t: \RR_+ \to \RR_+$ such that
\begin{itemize}[leftmargin=1.2cm]
  \item $\phi_{t_1 + t_2} = \phi_{t_1} \ast \phi_{t_2}$.
  \item For every $f \in \M$ and $t \geq 0$
  \[
    S_t(f) = \int_{0}^\infty T_s(f) \phi_t(s) \, d \, s.
  \]
\end{itemize}
Additionally, when the functions $\phi_t$ are of bounded variation and satisfy that $\lim_{s \to \infty} \phi_t(s) = 0$ for every $s$, we have 
\begin{equation}
  S_t = \int_0^\infty T_s \phi_t(s) \, d  s
      =  - \int_0^\infty T_s \bigg( \int_s^\infty d\phi_t(z) \bigg) \, d  s
      =  \quad - \int_0^\infty \bigg( \dashint_0^z T_s \bigg) z \, d\phi_t(z),
\end{equation}
where the signed measure $d \phi_t(z)$ is the Lebesgue-Stieltjes derivative of $\phi_t$. If 
\[
  \sup_{t \geq 0} \int_0^\infty z \, d |\phi_t| (z) < \infty,
\]
where $d |\phi_t| (z)$ is the total variation associated to the signed measure above, we have that the bilateral almost uniform convergence of $(S_s)_{s \geq 0}$ is implied by that of $(M_s(\T))_{t \geq 0}$.

One important class of subbordinated semigroups are the \emph{subordinated Poisson semigroups}. Since $\S = (S_t)_{t \geq 0}$ extends to $L_2(\M)$, we can use the spectral theorem to express $S_t$ as $e^{- t A}$, when $A$ is an unbounded self adjoint operator. A result that goes back at least to Stein \cite{Ste1970} implies that $P_s = e^{- s A^\frac12}$ is also a Markovian semigroup and the following subordination formula holds
\begin{equation}
  \label{eq:Subordination}
  P_s(f)
    = \frac1{2 \sqrt{\pi}}
    \int_0^\infty \underbrace{s e^{-\frac{s^2}{4 v}} v^{-\frac{3}{2}}}_{\phi_s(v)} S_v(f) \, d v.
\end{equation}
The formula above follows from an explicit computation of the inverse Laplace transform of $e^{- s \eta^\frac12}$. The name of Poisson stems from the fact that, when $S_t$ is taken to be the heat semigroup in $\RR^n$ or $\TT^n$, $P_s$ recovers the classical Poisson semigroup. A simple computation shows that the functions $\phi_s(v)$ satisfy
\[
  \sup_{s>0} \Big\| s \, \frac{\partial}{\partial \, s} \phi_s(v) \Big\|_1 < \infty.   
\]
As a consequence we obtain the following:

\begin{proposition}
  \label{prp:subbordinationBAU}
  Let $\S = (S_t)_{t \geq 0}$ be a Markovian semigroup and $P_s(f)$ its subordinated Poisson semigroup.
  For every $f \in L_1(\M)$ we have that $P_s(f) \to f$ bilaterally almost uniformly as $s \to 0$.
\end{proposition}

% Martingales
\subsection*{Martingales}
Here we are going to give basic facts and definitions regarding the theory of noncommutative martingale sequences, see also \cite{Pisier2016MartingaleBook,PisierXu1997Martingales} for more details. Let $\N \subset \M$ be two finite von Neumann algebras and assume that the inclusion $i$ is unital. Since $\tau (i(f)) = \tau(f)$, $i$ extends to the $L_1$-spaces $L_1(\N) \subset L_1(\M)$. Let us denote this extended inclusion by $i_1$. Its dual $i_1^\ast = \EE_\N: \M \to \N \subset \M$ is the conditional expectation onto $\N$, i.e. the map $\EE_\N$ is a surjective and completely positive projection that is furthermore $\N$-bimodular. 
Given a sequence of finite von Neumann subalgebras $(\M_m)_n$ of $\M$, we will say that they are nested if $\M_m \subset \M_{m+1}$. We will also assume in all of the examples that $(\M_m)_m$ approximates $\M$ is the sense that
\[
  \overline{{\bigcup_{n \geq 1} \M_n}^{\wast}} = \M.
\]
The sequence of algebras $(\M_m)_m$ is sometimes referred to as a filtration and the associated sequence of operators $f_m = \EE_m(f)$ as a martingale sequence. We will denote by $\C$ the union of all the algebras $\M_m$. Clearly $\S$ is norm dense in $L_p(\M)$ with $1 \leq p < \infty$. We also have that the analogue of Proposition \ref{prp:ErgodicNorm} holds for the weak-$\ast$ dense algebra $\S$ given by the union of $\M_n$.

%\begin{proposition}
%  \label{prp:MartingaleNorm}
%  Let $(\M,\tau)$ be a finite von Neumann algebra and $(\M_m)_m$, $(\EE_m)_m)$ and $\S$
%  be as above. We have that
%  \begin{enumerate}[leftmargin=1.2cm, label={\rm (\roman*)}, ref={\rm (\roman*)}]
%    \item \label{itm:MartingaleNorm.1} For every $x \in L_1(\M)$, $\| \EE_n(f) - f \|_1 \to 0$.
%    \item \label{itm:MartingaleNorm.2} For every $x \in \M$, $\EE_n(f) \to f$ in the weak-$\ast$ topology.
%    \item \label{itm:MartingaleNorm.3} There is a class $\S \subset \M$, weak-$\ast$ dense in $\M$ and
%    norm dense in $L_p$-spaces for $1 \leq p < \infty$ and inside $L \log^\alpha L(\M)$,
%    such that $\| \EE_n(f) - f \|_\infty$, for every $f \in \S$.
%  \end{enumerate}
%\end{proposition}

Cuculescu showed that the family $(\EE_n)$ is maximally of weak type $(1,1)$, see \cite{Cuculescu1971}. This result is a generalization of Doob's maximal inequality in the noncommutative setting. We have the following consequences, obtained through Theorems \ref{thm:Extrapolation} and \ref{thm:JunXuRealInterpolation}.

\begin{theorem}
  \label{thm:MartingalesLpMax}
  Let $(\M,\tau)$, $(\M_m)_m$ and $(\EE_m)_m$ be as above.
  \begin{enumerate}[leftmargin=1.2cm, ref={\rm (\roman*)}, label={\rm (\roman*)}]
    \item For every $1 < p < \infty$
    \begin{equation}
      \label{eq:MartingalesLpMax1}
      \big\| (\EE_n)_n: L_p(\M) \to L_p(\M; \ell_\infty) \big\|
      \, \lesssim \,
      \max \Big\{ 1, \Big( \frac1{p - 1} \Big)^2 \Big\}.
    \end{equation}
    \item \label{eq:MartingalesLpMax2}
      $\displaystyle{
      \big\| (\EE_n)_n: L \log^2 L(\M) \to L_1(\M; \ell_\infty) \big\|
      \, < \,
      \infty.}$
    \item As a consequence $\EE_n(f) \to f$ bilaterally almost uniformly
    for every $f \in L_1(\M)$.
  \end{enumerate}
\end{theorem}

% Free and tensor products of von Neumann algebras
\subsection*{Free products of von Neumann algebras}
It is easy to see that if $\phi_1: \M_1 \to \M_1$ and $\phi_2: \M_2 \to \M_2$ are completely positive normal maps over von Neumann algebras, their tensor product $\phi_1 \otimes \phi_2: \M_1 \weaktensor \M_2 \to \M_1 \weaktensor \M_2$ is also completely positive and normal. Furthermore, if $\tau_1$ and $\tau_2$ are traces, we have that the tensor product of two Markovian operators $T_1: \M_1 \to \M_1$ and $T_2: \M_2 \to \M_2$ extends  to a Markovian operator $T_1 \otimes T_2: \M_1 \weaktensor \M_2 \to \M_1 \weaktensor \M_2$. The same result holds for $\EE_1 \otimes \EE_2: \M_1 \weaktensor \M_2 \to \N_1 \weaktensor \N_2 \subset \M_1 \weaktensor \M_2$ whenever $\EE_i: \M_i \to \N_i \subset \M_i$ are conditional expectations. We are going to recall a result from \cite{Choda1996} that allows to do the same for free products. 

Let $(\M_i, \tau_i)$ for $i \in \{1,2\}$ be two finite von Neumann algebras and assume that they are cyclically represented inside Hilbert spaces $\M_i \subset \B(H_i)$. We will fix cyclic vectors $\xi_i \in H_i$. The reader can assume without loss of generality that $(H_i, \xi_i)$ is the GNS construction associated to $\tau_i$.
We will denote by $(H,\xi)$ the \emph{free product} of the pointed Hilbert spaces $(H_i,\xi_i)$ given by
\[
  H = \CC \xi \oplus \bigoplus_{n \geq 0} \, \bigoplus_{i_1 \neq i_2 \, ... \, \neq i_n} \,
              H_{i_1}^\circ \otimes H_{i_2}^\circ \otimes H_{i_3}^\circ \, ... \, \otimes H_{i_n}^\circ,
\]
where $\H_i^\circ = H_i \ominus \CC \xi_i = (\CC \xi_i)^\perp$. In order to define the reduced free product $\M_1 \ast \M_2$, notice that there are two unitary maps $V_i: H(i) \otimes_2 H_i \to H$, where each $H(i) \subset H$ is given by
\[
  H(i) = \CC \xi \oplus \bigoplus_{n \geq 0} \, \bigoplus_{i \neq i_2 \, ... \, \neq i_n} \,
              H_{i}^\circ \otimes H_{i_2}^\circ \otimes H_{i_3}^\circ \, ... \, \otimes H_{i_n}^\circ,
\]
see \cite{Voiculescu1992Book} for the details. The unitaries above allow to define two faithful representations $\pi_i: \M_i \to B(H)$ for $i \in \{1,2\}$ given by 
\[
  \pi_i(a) = V_i \, (\1 \otimes a) \, V_i^\ast.
\]
We will denote by $\pi = \pi_1 \ast \pi_2$ the $C^\ast$-algebra representations defined over the universal free product representation of $\M_1$ and $\M_2$ as $C^\ast$-algebras. The image of $\pi$ is a quotient of such universal $C^\ast$- algebra. We will denote by $\M_1 \ast \M_2$ the von Neumann algebra given by the weak-$\ast$ closure on the image of $\pi$. The vector $\xi \in \H$ defines a natural trace that we will denote by $\tau_1 \ast \tau_2$. Notice that if we denote by $\M_i^\circ$ the weak-$\ast$ closed linear subsbpace given by
\[
  \M_i^\circ = \big\{ f \in \M_i : \tau_i(f) = 0 \big\},
\]
then we can define a weak-$\ast$ dense unital $\ast$-algebra $D$ inside $\M_1 \ast \M_2$ as
\[
  D = \CC \1 \oplus \bigoplus_{n \geq 0} \, \bigoplus_{i_1 \neq i_2 \, ... \, \neq i_n} \,
              \M_{i_1}^\circ \otimes \M_{i_2}^\circ \otimes \M_{i_3}^\circ \, ... \, \otimes \M_{i_n}^\circ,
\]
where the direct sums and tensor products are understood in an algebraic sense. The multiplication is given by associativity and the following rule for pure tensors 
\[
  \begin{split}
    \big( x_1 \otimes & x_2 \otimes \, ... \otimes x_n \big) \cdot \big( y_1 \otimes y_2 \, ... \otimes y_m \big)\\
    & =
    \begin{cases}
      x_1 \otimes x_2 \otimes \, ... \otimes x_n \otimes y_1 \otimes y_2 \, ... \otimes y_m
        & \mbox{ when } x_n \in \M_{i_n}^\circ, \, y_1 \in \M_{j_1}^\circ \mbox{ and } i_n \neq j_1 \\
     (x_1 \otimes x_2 \otimes \, ... \otimes x_{n-1}) \cdot \big( (x_n \, y_1)^\circ \, + \, \tau(x_n \, y_1) \1 \big) \otimes y_2 \, ... \otimes y_m & \mbox{ otherwise }
    \end{cases}
  \end{split}
\]
In the formula above $(x_n \, y_1)^\circ$ is notation for $x_n \, y_1 \, - \, \tau(x_n \, y_1)$. Given two unital maps $T_1$ and $T_2$ preserving the states $\tau_1$ and $\tau_2$ respectively, we can define the map $T_1 \ast T_2: D \to D$ by linear extension of 
\[
  (T_1 \ast T_2) \big( x_1 \otimes x_2 \otimes \, ... \otimes x_n \big)
  = 
  T_{i_1}(x_1) \otimes T_{i_2}(x_2) \otimes \, ... \otimes T_{i_n}(x_n),
\]
where $x_1 \otimes x_2 \otimes \, ... \otimes x_n \in \M_{i_1}^\circ \otimes \M_{i_2}^\circ \otimes \M_{i_3}^\circ \, ... \, \otimes \M_{i_n}^\circ$ with $i_1 \neq i_2 \, ... \, \neq i_n$. The following result states that we can extend the map $T_1 \ast T_2$ to the whole $\M_1 \ast \M_2$ when the $T_i$ are normal u.c.p. and trace preserving.

\begin{theorem}[{\bf \cite[Proposition 2.1/Corollary 2.2]{Choda1996}}]
  \label{thm:FreeProductExtension}
  \
  \begin{enumerate}[leftmargin=1.2cm, ref={\rm (\roman*)}, label={\rm (\roman*)}]
    \item \label{itm:FreeProductExtension.1}
    Let $T_i: \M_i \to \M_i$ for $i \in \{1,2\}$ be Markovian maps, the map $T_1 \ast T_2$
    extends to a Markovian map over $(\M_1 \ast \M_2, \tau_1 \ast \tau_1)$. Furthermore $T_1 \ast T_2$ is symmetric
    when the maps $T_i$ are.
    \item \label{itm:FreeProductExtension.2}
    Let $\N_i \subset \M_i$ be unital von Neumann subalgebras and let $\EE_i: \M_i \to \N_i \subset \M_i$
    be their asssociated conditional expectations, $\EE_1 \ast \EE_2: \M_1 \ast \M_2 \to \N_1 \ast \N_2 \subset \M_1 \ast \M_2$ is a conditional expectation. 
  \end{enumerate}
\end{theorem}

% Group algebras
\subsection*{Group von Neumann algebras}
We end this section with another interesting example of finite von Neumann algebra that will be used in Section \ref{sec:F2}. Let $\Gamma$ be a discrete group and $\lambda: \Gamma \to \B(\ell_2 \Gamma)$ the left regular  representation, that is, the one given by $\lambda_g \delta_h = \delta_{g h}$, where $(\delta_h)_{h \in \Gamma}$ is the canonical orthogonal base of $\ell_2(\Gamma)$. We define the group von Neumann algebra $\L \Gamma$ by
\[
  \L \Gamma = \wstspan \{ \lambda_g \}_{g \in \Gamma} = \{\lambda_g \}_{g \in \Gamma}'' \subset \B(\ell_2 \Gamma).
\]
$\L \Gamma$ has a natural faithful trace $\tau: \L \Gamma \to \CC$ given by $\tau(f) = \langle \delta_e, f \, \delta_e \rangle$, see \cite{Ped1979}. It is worth noting that each $f \in \L \Gamma$ can be expressed as a sum
\[
   f = \sum_{g \in \Gamma} \widehat{f}(g) \, \lambda_g,
\]
where $\widehat{f}(g) = \tau(f \, \lambda_g^\ast)$. The series above converges in the weak-$\ast$ topology. The notation $\widehat{f}$ is reminiscent of the case of Abelian groups, where $\L \Gamma = L_\infty(\widehat{\Gamma})$, with $\widehat{\Gamma}$ being the Pontryagin dual of $\Gamma$. In this setting, $\widehat{f}$, as defined above, coincides with the Fourier transform for Abelian groups, see \cite[Chapter 4]{Foll1995}. 

In this article we will work over the algebra of the free group on two generators $\FF_2$. Let us denote the generators by $a$ and $b$ and let $\L \FF_2$ be its reduced group von Neumann algebra. Over $\FF_2$ there is a natural length $\omega \mapsto |\omega|$ given by
\[
  |\omega|  = \min \big\{ n : \omega = s_1 s_2 ... s_n \mbox{ with } s_i \in \{ a, a^{-1}, b, b^{-1} \} \big\}.
\]
It was shown in \cite{Haa1978} that the word length generates a Markovian and symmetric semigroup of operators $\S = (S_t)_{t \geq 0}$. Indeed, given any element $f \in \L \FF_2$ we define the maps $S_t: \L \FF_2 \to \L \FF_2$
given by
\[
  S_t \bigg( \sum_{\omega \in \FF_2} \widehat{f}(\omega) \, \lambda_\omega \bigg)
  = \sum_{\omega \in \FF_2} \widehat{f}(\omega) \, e^{- t |\omega|} \, \lambda_\omega.
\]

\section{\bf The limit maximal function \label{scn:limsup}}
In this section we will introduce the $L_p$-norm of the $\limsup$ operator for von Neumann algebras. Let us introduce the linear (nonclosed) space of bilaterally almost finite sequences $\L_\loc \subset L_p(\M;\ell_\infty) \subset L_1(\M;\ell_\infty)$. We will say that a sequence $(f_n)_n \in \L_\loc$ iff for every $\epsilon > 0$ there is a projection $e \in \Prj(\M)$ and an integer $N$ such that 
\begin{itemize}[leftmargin=1.2cm]
  \item $\tau(e) < \epsilon$.
  \item $e^\perp \, f_n \, e^\perp = 0$, for every $n > N$.
\end{itemize}
When working in $L_p(\M;\ell_\infty)$, the second condition can be interpreted by saying that $(e^\perp \, f_n \, e^\perp)_n \in L_p(\M;\ell_\infty^N) = L_{2p}(\M) \, (\ell_\infty^N \weaktensor \M) \, L_{2p}(\M)$.

It is relatively easy to see that $\L_\loc$ is a linear space. Indeed, it is trivial that $\L_\loc$ is closed under scalar products. To see that the space is closed under sums, assume that $(f_n)_n$ and $(g_n)_n$ are in $\L_\loc$ and take $\epsilon > 0$, by definition we can find $p,q \in \Prj(\M)$ and $N_1, N_2 \in \N$ such that 
\begin{itemize}[leftmargin=1.2cm]
  \item $\tau(p) < \epsilon/2$ and $\tau(q) < \epsilon / 2$.
  \item $p^\perp \, f_n \, p^\perp = 0$, for every $n > N_1$ and $q^\perp \, g_n \, q^\perp = 0$, for every $n > N_2$
\end{itemize}
we have that $e = p \vee q \in \Prj(\M)$ satisfies that $\tau(p \vee q) < \tau(p) + \tau(q) < \epsilon$ and that
$e^\perp = p^\perp \wedge q^\perp$, therefore
\[
  e^\perp \, (f_n + g_n) \, e^\perp
  = e^\perp \, f_n \, e^\perp + e^\perp \, g_n \, e^\perp
  = e^\perp \, (p^\perp \, f_n \, p^\perp) \, e^\perp + e^\perp \, (q^\perp \, g_n \, q^\perp) \, e^\perp = 0,
\]
for $n > \max \{ N_1, N_2 \}$.

\begin{definition}
  \label{def:Loclimsup}
  We define the space $L_p^\loc(\M;c_0)$ as the closure of $\L_\loc$ inside $L_p(\M;\ell_\infty)$.
  Similarly, we define
  \begin{equation}
    L_p^\loc(\M;\ell_\infty / c_0) = \frac{L_p(\M;\ell_\infty)}{L_p^\loc(\M;c_0)}
  \end{equation}
\end{definition}
Observe that given $(f_n)_n \in L_p(\M;\ell_\infty)$ the quotient map
$q: L_p(\M;\ell_\infty) \to L_p^\loc(\M;\ell_\infty/c_0)$ induces a seminorm on $L_p(\M;\ell_\infty)$ after composing with the norm of $L_p^\loc(\M;\ell_\infty/c_0)$, we will denote that seminorm as follows
\[
  \Big\| {\limsup_{n \to \infty}}^{+} f_n \Big\|_p := \| q (f_n) \|_{L_p^\loc[\ell_\infty/c_0]},
\]
and we will also omit the quotient map $q$ when no ambiguity can arise. Notice also that, as in the case of the supremum, the $\limsup^+$ is purely a notation. 

\begin{remark}
  \label{rmk:WhyNotQuotient}
  %\emph{
  Recall that, given a sequence of scalars $(a_n)_n$, we have
  \[
    \limsup_{n \to \infty} |a_n| = \| a_n \|_{\ell_\infty / c_0}.
  \]
  Therefore, it would be tempting to define the $L_p$-norm of the limsup as the seminorm associated with the quotient
  \[
    \frac{L_p(\M;\ell_\infty)}{L_p(\M;c_0)}.
  \]  
  Nevertheless, in the case of $p = 1$, such definition does not recovers the $L_1$-norm of the limsup, not even in the Abelian case. Indeed, notice that, for the sequence $(f_n)_n$ over $L_1([0,1])$, given by $f_n(t) = n \, \1_{[0,1/n]}(t)$, we have that
  \begin{eqnarray*}
    \big\| (f_n)_n \big\|_{\frac{L_1[\ell_\infty]}{L_1[c_0]}}
     & = & \inf_{N} \big\{ \big\| (f_n)_{n > N} \big\|_{L_1[\ell_\infty]} \big\}\\
     & = & \inf_{N} \Big\{ \Big\| \sup_{n > N} n \1_{[0,1/n]} \Big\|_1 \Big\}
     \quad = \quad \inf_{N} \int_0^{\frac1{N+1}} \Big\lceil \frac1{t} \Big\rceil \,d t = \infty.
  \end{eqnarray*}
  While the limsup of the sequence $f_n$ is $0$ almost everywhere. In this case, the sequence $f_n$ does not converge in the $L_1$-norm, but there are sequences of functions $f_n \in L_1([0,1])$ such that $f_n$ converge both in $L_1$-norm and almost everywhere and yet the seminorm associated to $L_1[\ell_\infty]/L_1[c_0]$ diverges. Such family of functions can be constructed by choosing an integrable function $f \in L_1([0,1])$ such that their truncated Hardy-Littlewood maximal function
  \[
    M(f)(t)
      = \sup_{0 < r \leq 1} M_r(f)(t)
      = \sup_{0 < r \leq 1} \frac1{r} \int_{t}^{t + r} |f(s)| \, d s
  \]
  is not in $L_1$ and then taking $f_n$ as a sequence of averages $M_{r_n}(f)$,
  with $r_n \to \infty$.
  
  To understand what is happening, recall that the Banach space functor $E \mapsto L_1(\Omega) \otimes_{\pi} E \cong L_1(\Omega; E)$, where $\otimes_\pi$ is the projective tensor product in the category of Banach spaces \cite{Ryan2013},  is both a projective and an injective functor. Therefore, we have
  \[
    L_1(\Omega) \otimes_\pi \ell_\infty/c_0
    =
    \frac{L_1(\Omega) \otimes_\pi \ell_\infty}{L_1(\Omega) \otimes_\pi c_0},
  \]
  for a measure space $\Omega$. But the nonseparability of $\ell_\infty/c_0$ makes $L_1(\Omega) \otimes_\pi \ell_\infty/c_0$ strictly smaller than the space of all measurable functions $F: \Omega \to \ell_\infty/c_0$ such that
  \[
    \theta \mapsto \| F(\theta) \|_{\ell_\infty/c_0}
    = \limsup_{n \to \infty} |f_n(\theta)|,
  \]
  is integrable, where $(f_n(\theta))_n$ is a representative of $F(\theta)$. This is due to the fact the equivalence between $L_1(\Omega) \otimes_\pi E$ and the measurable $E$-valued functions with integrable $E$-norm requires using Pettis' theorem, see \cite[Section 2.3]{Ryan2013}, which may fail in the nonseparable case.
  %}
\end{remark}

There is an alternative characterization of the norm of $L_p^\loc(\M;\ell_\infty/c_0)$ that would be useful afterwards

\begin{proposition}
  \label{prp:AltNorm}
  We have that
  \begin{equation}
    \label{eq:Altnorm}
    \Big\| {\limsup_{n \to \infty}}^+ f_n \Big\|_{p}
    = 
    \sup_{\epsilon > 0}
    \, \inf_{\substack{ e \in \Prj(\M) \\ \tau(e) < \epsilon }}
    \, \inf_{0 < N} \big\| (e^\perp \, f_n \, e^\perp)_{N < n} \big\|_{L_p[\ell_\infty]}
  \end{equation}
\end{proposition}

\begin{proof}
  By definition the $\limsup^+$ is given by 
  \[
    \Big\| {\limsup_{n \to \infty}}^+ f_n \Big\|_p
      = \inf \Big\{ \| (f_n - h_n)_n \|_{L_p[\ell_\infty]} : (h_n)_n \in L_p^\loc(\M;c_0) \Big\}
  \]
  but, by the definition of $L_p^\loc(\M;c_0)$, the linear subspace $\L_\loc$ is dense,
  therefore we can take the infimum of the right hand side in $(h_n)_n \in \L_\loc$. The space $\L_\loc$ can be expressed as 
  \[
    \L_\loc
    =
    \bigcap_{\epsilon > 0} \,
    \bigcup_{\substack{e \in \Prj(\M) \\ \tau(e) < \epsilon}} \,
    \bigcup_{N \in \NN} \,
    \big\{ (h_n)_n : e^\perp \, h_n \, e^\perp = 0, \forall n > N \big\}.
  \]
  Using the monotonicity of the infimum over the index sets we obtain that
  \[
    \Big\| {\limsup_{n \to \infty}}^+ f_n \Big\|_p
    = \sup_{\epsilon > 0} \,
      \inf_{\substack{e \in \Prj(\M) \\ \tau(e) < \epsilon}} \,
      \inf_{N} \,
      \inf \Big\{ \, \| (f_n - h_n)_n \|_{L_p[\ell_\infty]}
                  : e^\perp \, h_n \, e^\perp = 0, \forall n > N \Big\}.
  \]
  Fix $\epsilon > 0$, $e$ and $N$. In the innermost infimum we can choose $h_n$ as 
  \[
    h_n =
    \begin{cases}
      f_n & \mbox{ when } n \leq N\\
      e \, f_n \, e + e \, f_n \, e^\perp + e^\perp \, f_n \, e & \mbox{ when } n > N
    \end{cases}
  \]
  and a calculation gives that
  \[
    \| f_n - h_n \|_{L_p[\ell_\infty]} = \big\| (e^\perp \, f_n \, e^\perp)_{n>N} \big\|_{L_p[\ell_\infty]}.
  \]
  This implies the inequality $(\leq)$ in equation \eqref{eq:Altnorm}.
  %The other is proved similarly.
  For the other direction we will need to use the module properties \eqref{eq:bimoduleProp} and \eqref{eq:monotonicLpLinfty}. Fix $\epsilon > 0$, $e \in \Prj(\M)$ with $\tau(e) < \epsilon$ and $N$ an integer. We have that
  \[
    \big\| (e^\perp \, f_n \, e^\perp)_{n > N} \big\|_{L_p[\ell_\infty]}
    =
    \big\| (e^\perp \, (f_n - h_n) \, e^\perp)_{n > N} \big\|_{L_p[\ell_\infty]}
  \]
  for every $h_n$ satisfying that $e^\perp \, h_n \, e^\perp \in L_p[\ell_\infty^N]$.
  Therefore, we can take an infimum in such identity to obtain that
  \begin{equation*}
    \begin{array}{r>{\displaystyle}cl>{\displaystyle}r}
      \big\| (e^\perp \, f_n \, e^\perp)_{n > N} \big\|_{L_p[\ell_\infty]}
        & = & \inf \Big\{ \, \big\| (e^\perp \, (f_n - h_n) \, e^\perp)_{n > N} \big\|_{L_p[\ell_\infty]} : e^\perp \, h_n \, e^\perp = 0, \forall n > N \Big\} & \\
        & \leq & \inf \Big\{ \, \big\| (e^\perp \, (f_n - h_n) \, e^\perp)_{n} \big\|_{L_p[\ell_\infty]} : e^\perp \, h_n \, e^\perp = 0, \forall n > N \Big\} & \mbox{ by \eqref{eq:monotonicLpLinfty}}\\
        & \leq & \inf \Big\{ \, \big\| (f_n - h_n)_{n} \big\|_{L_p[\ell_\infty]} : e^\perp \, h_n \, e^\perp = 0, \forall n > N \Big\} & \mbox{ by \eqref{eq:bimoduleProp}}
    \end{array}
  \end{equation*}
  and that concludes the proof.
\end{proof}

Recall that, whenever needed, it is possible to describe a natural operator space structure for the spaces $L_p^\loc(\M;\ell_\infty/c_0)$. Observe that the functor $E \mapsto S_p^m[E]$ is projective, therefore, we can identify the following spaces
\[
  S_p^m \Big[ \frac{L_p(\M;\ell_\infty)}{L_p^\loc(\M;c_0)} \Big]
  \, = \,
  \frac{S_p^m \big[ L_p(\M;\ell_\infty) \big]}{S_p^m \big[ L_p^\loc(\M;c_0) \big]}
  \, = \,
  \frac{L_p(M_m \otimes \M;\ell_\infty)}{S_p^m \big[ L_p^\loc(\M;c_0) \big]}
\]
isometrically. We can represent $L^\loc_p(\M;c_0)$ as the closure of
\[
  \bigcap_{\epsilon > 0} \, \bigcup_{\substack{e \in \Prj(\M)\\ \tau(e) < \epsilon}} \, \bigcup_{N \in \NN} \, e^\perp \, L_p(\M;\ell_\infty^N) \, e^\perp
\]
But using the fact that 
\[
  S_p^m[ e^\perp \, L_p(\M; \ell_\infty^N) \, e^\perp ] = e^\perp \, L_p(M_m \otimes \M; \ell_\infty^N) \, e^\perp,
\]
we obtain that
\[
  S_p^m \big[ L_p^\loc(\M; \ell_\infty/c_0) \big]
  \, = \,
  L_p^\loc(M_m \otimes \M; \ell_\infty/c_0).
\]
The following two propositions link the boundedness of the noncommutative
$\limsup$ with the bilateral almost uniform convergence.

\begin{proposition}
  \label{prp:BAUImpLimsup}
  Let $(f_n)_n \subset L_p(\M)$ is a sequence such that $f_n \to f$ bilaterally almost uniformly we have that
  \begin{equation}
    \label{eq:BAUImpLimsup}
    \Big\| {\limsup_{n \to \infty}}^+ f_n \Big\|_p = \| f \|_p.
  \end{equation}
\end{proposition}

\begin{proof}
  The proof is almost trivial. Start by noticing that in the case of finite von Neumann algebras we have that
  \[
    \| (f_n)_n \|_{L_p[\ell_\infty]} \leq \, \| (f_n)_n \|_{\ell_\infty \weaktensor \M}.
  \]
  Therefore if $g_n \to 0$ b.a.u. we have that
  \[
    \sup_{\epsilon > 0} \,
    \inf_{\substack{e \in \Prj(\M) \\ \tau(e) < \epsilon}} \,
    \inf_{N \in \NN} \,
    \big\| (e^\perp \, g_n \, e^\perp)_{n > N} \big\|_{L_1[\ell_\infty]}
    \leq 
    \sup_{\epsilon > 0} \,
    \inf_{\substack{p \in \Prj(\M) \\ \tau(p) < \epsilon}} \,
    \inf_{N \in \NN} \,
    \big\| (p^\perp \, g_n \, p^\perp)_{n > N} \big\|_{\ell_\infty \weaktensor \M}
  \]
  and the last expression tends to $0$. Taking $g_n = f - f_n$ gives
  \eqref{eq:BAUImpLimsup}.
\end{proof}

\begin{remark}
  \label{rmk:semifinite} 
  In the previous proposition we impose the condition that $(\M,\tau)$ must be finite.
  There are pathologies that can occur in the semifinite case.
  For instance the sequence $f_n = \1_{[n,n+1]}$ converge pointwise
  to zero and is uniformly bounded inside $L_1(\RR_+)$.
  Nevertheless, if $\epsilon > 0$ is a finite number and $E \subset \RR_+$
  is a measurable set such that $m(E) < \epsilon$,
  then $(\1_{E^c} f_n)_{N < n}$ is not in $L_1(\RR_+;\ell_\infty)$ for any $N$.
\end{remark}

\begin{proposition}
  \label{prp:LimsupImpBAU}
  If $(S_n)_n$ is a bounded family of
  operators  $S_n: L_1(\M) \to L_1(\M)$ such that
  \begin{enumerate}[leftmargin=1.2cm,label={\rm (\roman*)}, ref={\rm (\roman*)}]
    %\item \label{itm:LimsupImpBAU.1} There is a $F:L_1(\M) \to L_1(\M)$ such that $\|S_n(f) - F(f)\|_1 \to 0$, $\forall f \in L_1(\M)$.
    \item \label{itm:LimsupImpBAU.2} 
    There is a bounded normal $F: \M \to \M$
    $\|S_n(f) - F(f)\|_\infty \to 0$, $\forall f \in \S$, a class $\S \subset L_1(\M) \cap \M$
    dense in $L_\Phi(\M)$   
    \item \label{itm:LimsupImpBAU.3} For every $f \in L_\Phi(\M)$, we have
    \[
      \Big\| {\limsup_{n \to \infty}}^{+} \, S_n(f) \Big\|_{1} \lesssim \| f \|_{L_\Phi}.
    \]
  \end{enumerate}
  Then, for every $f \in L_\Phi$, $S_n(f) \to F(f)$ bilaterally almost uniformly.
\end{proposition}

\begin{proof}
  Fix, $f \in L_\Phi(\M)$. We want to prove that, for every $\epsilon > 0$ and $\delta > 0$, there is an integer $N$ and projection $e \in \Prj(\M)$ such that $\tau(e) < \epsilon$ and 
  \[
    \big\| e^\perp \, (S_n(f) - F(f)) \, e^\perp \big\|_\infty < \delta
  \]
  for every $n > N$.
  Denote $S_n - F$ by $T_n$ and denote by $C$ the optimal quantity satisfying that
  \[
    \Big\| {\limsup_{n \to \infty}}^+ T_n(f) \Big\|_1 \leq C \, \| f \|_\Phi.
  \]
  By the density of $\S$ inside $L_\Phi(\M)$ we can choose $g \in \S$ such that
  $\| f - g \|_\Phi < \delta_0$, where $\delta_0 < \delta \, \epsilon / (8 C)$.
  Now, we have that
  \[
    T_n(f) = T_n(f - g) + T_n(g)
  \]
  and the second terms tends to $0$ in the $L_\infty$-norm,
  therefore there is an integer $N_1$ such that $\| T_n(g) \|_\infty < \delta$ for every $n > N_1$.
  For the first term we do have that 
  \[
    \Big\| {\limsup_{n \to \infty}}^+ T_n(f - g) \Big\|_1 \leq C \, \delta_0.
  \]
  But by \eqref{eq:L1Linfty} we do have that there is a projection $p \in \Prj(\M)$
  such that $\tau(p) < \epsilon / 2$  and an integer $N_0$ such that
  \[
    \big\| p^\perp \, (T_n(f - g))_{n > N_0} \, p^\perp \big\|_{L_1[\ell_\infty]} \leq 2 C \delta_0.
  \]
  Therefore, we can factor the term above as $p^\perp \, (T_n(f - g)) \, p^\perp = \alpha \, h_n \, \beta$, 
  where $\| \alpha \|_2 = \| \beta \|_2 = \sqrt{2 C \delta_0}$ and $\| h_n \|_{\ell_\infty \otimes \M} = 1$
  Fix a $\lambda > 0$, to be determined later, and define the projections
  \[
    r = \1_{[\lambda,\infty)}(\alpha \alpha^\ast)
    \quad \mbox{ and } \quad 
    q = \1_{[\lambda,\infty)}(\beta^\ast \beta).
  \]
  By, Chebishev's inequality we have that
  \[
    \tau(r) \leq \frac{\| \alpha \|_2^2}{\lambda}
    \quad \mbox{ and } \quad
    \tau(q) \leq \frac{\| \beta \|_2^2}{\lambda}.
  \]
  Taking $\lambda = \delta$ gives that $e = r \vee q \vee p$ satisfies that
  \[
    \tau(e) \leq \frac{\epsilon}{2} + \frac{4 C \delta_0}{\lambda} = \frac{\epsilon}{2} + \frac{\delta \epsilon}{ 2 \, \delta} = \epsilon
  \]
   but we also have that
   \begin{eqnarray*}
     \big\| e^\perp \, T_n(f - g) \, e^\perp \big\|_\infty
       & = & \big\| e^\perp \, p^\perp \, T_n(f - g) \, p^\perp \, e^\perp \big\|_\infty\\
       & = & \big\| e^\perp \, \alpha \, h_n \, \beta \, e^\perp \big\|_\infty\\
       & \leq & \big\| e^\perp \, \alpha \big\|_\infty \,
                \Big( \sup_n \big\| h_n  \big\|_\infty \Big) \,
                \big\| \beta \, e^\perp \big\|_\infty \quad < \quad \delta^\frac12 \delta^\frac12.
  \end{eqnarray*}
  Taking $N = \max \{N_0, N_1\}$ and $e \in \Prj(\M)$ gives that
  \[
    \big\| e^\perp \, (F(f) - S_n(f)) \, e^\perp \big\|_\infty < \delta + \delta
  \]
  for every $n > N$ and that ends the proof.
\end{proof}

Throughout the text, we will need to work also with limits in more than one index.
We will use the convention that, for scalar sequences $(a_{n,m})_{n,m}$
\begin{equation}
  \limsup_{n,m \to \infty} |a_{n,m}| = \inf_{0 \leq N,M} \sup_{n \geq N, m \geq M} |a_{n,m}|.
\end{equation}
This notion of $\limsup$ is chosen in order to avoid the case in which the 
pair $(n,m)$ tends to infinite, in the sense of escaping of every finite set,
while one of the coordinates $n$ or $m$ does not. For any given any pair
of sequences $n(j)$ and $m(j)$, both of them tending to infinity, we have that 
\[
  \limsup_{j \to \infty} |a_{n(j),m(j)}| \leq \limsup_{n,m \to \infty} |a_{n,m}|
\]
and the inequality is tight in the sense that we can make those quantities the
same by choosing adequate subsequences $n(j)$ and $m(j)$ as above. That $\limsup$ can be
represented as a quotient norm as follows. Let $L_{N,M} \subset \ell_\infty(\NN \times \NN)$
be the following space
\[
  L_{N,M} = \ell_\infty^N \weaktensor \ell_\infty + \ell_\infty \weaktensor \ell_\infty^M
\]
We define $c_\Delta(\NN \times \NN)$ as the norm closure of all $L_{N,M}$
\begin{equation*}
  c_\Delta(\NN \times \NN)
  = \overline{{\spn}^{\| \cdot \|_\infty}} \big\{ L_{N,M} : N, M \big\}
  \subset \ell_\infty(\NN \times \NN).
\end{equation*}
We have that
\begin{equation}
  \limsup_{n, m \to \infty} |a_{n,m}|
  =
  \| (a_{n,m})_{n,m} \|_{\ell_\infty(\NN \times \NN) / c_\Delta(\NN \times \NN)}
\end{equation}

In a way analogous to what we did before. We are going to define a nonclosed linear subspace $\L^\Delta_\loc \subset L_p(\M;\ell_\infty(\NN \times \NN))$ given those sequences $(f_{n,m})_{n,m}$ such that, for every $\epsilon > 0$ there is a projection $p \in \Prj(\M)$ and some integers $N,M$ such that
\begin{itemize}[leftmargin=1.2cm]
  \item $\tau(p) < \epsilon$
  \item $p^\perp \, f_{n,m} \, p^\perp = 0$ whenever $n > N$ and $m > M$.
\end{itemize}
Observe that the second condition can be expressed as saying that $p^\perp \, f_{n,m} \, p^\perp \in L_p(\M;L_{N,M})$.
We define the space $L_p^\loc(\M;c_\Delta(\NN \times \NN))$ as the norm closure of $\L^\Delta_\loc$ inside $L_p(\M;\ell_\infty)$. We will omit the dependency on the iterated products of the integers $\NN \times \NN$, when they can be understood from the context. We define
\[
  L_p^\loc(\M; \ell_\infty/c_\Delta)
  = \frac{L_p(\M; \ell_\infty)}{L_p^\loc(\M;c_\Delta)}
\]
We will denote the seminorm induced in $L_p(\M;\ell_\infty)$ by the above quotient as
\[
  \Big\| {\limsup_{n,m \to \infty}}^+ f_{n,m} \Big\|_p := \big\| (f_{n,m})_{n,m} \big\|_{L_p^\loc[\ell_\infty/c_\Delta]}
\]
Similarly, we can define multi-indexed versions of the above $\limsup$ and express them as a quotient. The operator space structures of the spaces $L_p^\loc(\M;\ell_\infty/c_\Delta)$ are defined in the same way. The analogues of Proposition \ref{prp:AltNorm} also hold true, as well as the connections with the bilateral almost uniform convergence in more than one variable in Propositions \ref{prp:BAUImpLimsup} and \ref{prp:LimsupImpBAU}.

In the next section we will use the following lemma, whose proof is immediate and left to the reader.

\begin{lemma}
  \label{lem:MultiInd}
  \
  \begin{enumerate}[leftmargin=1.2cm, label={\rm (\roman*)}, ref={\rm (\roman*)}]
    \item Let $n(j)$ and $m(j)$ be sequences of integers converging to infinity, then
    \begin{equation}
      \Big\| {\limsup_{j \to \infty}}^+ f_{n(j), m(j)} \Big\|_p
      \leq
      \Big\| {\limsup_{n, m \to \infty}}^+ f_{n, m} \Big\|_p.
    \end{equation}
    \item When $(f_{n})$ only depends on one of the indices, we have that
    \begin{equation}
      \Big\| {\limsup_{n, m \to \infty}}^+ f_{n} \Big\|_p
      =
      \Big\| {\limsup_{n \to \infty}}^+ f_{n} \Big\|_p.
    \end{equation}
  \end{enumerate}
\end{lemma}

\begin{remark}
  \label{rmk:asymmetric}
  In the same way in which Propositions \ref{prp:BAUImpLimsup} and \ref{prp:LimsupImpBAU} connect the boundedness of the $\limsup^+$ with the bilateral almost uniform convergence it is possible to introduce asymmetric versions of the space $L_p^\loc(\M;\ell_\infty)$ that are connected to the column and row almost uniform convergence. Start recalling the definition of the asymmetric $L_p[\ell_\infty]$-spaces from \cite{JunXu2007, JunPar2010}, see also \cite[3.1.2]{Defant2011}, given by
  \begin{eqnarray*}
    L_p(\M;\ell_\infty^c) & = & \big( \M \weaktensor \ell_\infty \big) \, L_p(\M)\\
    L_p(\M;\ell_\infty^r) & = & L_p(\M) \, \big( \M \weaktensor \ell_\infty \big),
  \end{eqnarray*}
  where, as before, the norms are given by taking the infima over all decompositions of $(f_n)$ as $f_n = h_n \, \alpha$ or $f_n = \alpha \, h_n$ of $\| \alpha \|_p \sup_n \| h_n \|_\infty$. Those spaces have been studied for their connection with the column/row almost uniform convergence. Let us focus our attention in the column case since the other one follows symmetrically. We can define a linear non-closed subspace $\L_c \subset L_1(\M;\ell_\infty^c)$ given by all sequences $(f_n)_n$ such that for every $\epsilon > 0$, there is a projection $e \in \Prj(\M)$ and an integer $N$ such that
  \begin{itemize}[leftmargin=1.2cm]
    \item $\tau(e) < \epsilon$
    \item $f_n \, e^\perp = 0$, for every $n > N$
  \end{itemize}
  The closure of $\L_c$ will be denoted by $L_p^\loc(\M;c_0^c) \subset L_p(\M;\ell_\infty^c)$.
  We define the quasi-Banach space
  \[
    L_p^\loc(\M;\ell_\infty^c/c_0^c) = \frac{L_p(\M;\ell_\infty^c)}{L_p(\M;c_0^c)}.
  \]
  The analogues of Propositions \ref{prp:BAUImpLimsup} and \ref{prp:LimsupImpBAU}  hold in this case. Indeed, if $(\M,\tau)$ is a finite von Neumann algebra, it holds that
  \begin{equation*}
    f_n \longrightarrow f \mbox{ column almost uniformly }
    \quad \Longrightarrow \quad 
    \big\| (f_n)_n \big\|_{L_p^\loc[\ell_\infty^c/c_0^c]} \leq \| f \|_p.
  \end{equation*}
  Similarly, if in Proposition \ref{prp:LimsupImpBAU}, we change the third assumption by
  \[
    \big\| (T_n)_n: L_\Phi(\M) \to L_p^\loc(\M;\ell_\infty^c / c_0^c) \big\| < \infty.
  \]
  Then, the consequence is that $T_n(f) \to F(f)$ column almost uniformly.
\end{remark}

\section{\bf The Jessen-Marcinkiewicz-Zygmund inequality \label{sct:JMZ}}

In this section we will prove one of the main results of the article, Theorem \ref{thm:MainTool}. Despite following relatively easily from the theory that we have already developed, that theorem is the main ingredient of Corollary \ref{cor:MultiMartingale}. We will also explore applications to multiparametric ergodic theory and asymmetric, i.e. column/row, versions of those results.

\begin{proof}[{\bf Proof (of Theorem \ref{thm:MainTool}.)}]
  Since all the operators involved $B_m$ and $A_n$ are assumed to be positivity preserving, it is enough to prove the result within the positive cone. Fix $f \in L_\Phi(\M)$ with $f \geq 0$. By Proposition \ref{prp:AltNorm}, we have that the
  norm of $L_1^\loc(\M;\ell_\infty/c_0)$ is monotone with respect to the natural
  order inherited by $L_1(\M;\ell_\infty)$. Using assumption \ref{itm:MainTool.1}
  as well as characterization \eqref{eq:L1Linfty} we have that for every $\epsilon > 0$, there is a $h \in L_1(\M)$ such that
  \[
    B_m(f) \leq h, \mbox{ for every } m
  \]
  and 
  \[
    \| h \|_1 \leq (1+\epsilon) \, \big\| (B_m)_m: L_\Phi(\M) \to L_1(\M;\ell_\infty) \big\| \| f \|_\Phi.
  \]
  The monotonicity and Lemma \ref{lem:MultiInd} gives that
  \[
    \Big\| {\limsup_{n,m \to \infty}}^{+} A_n \circ B_m (f) \Big\|_1 
    \leq 
    \Big\| {\limsup_{n,m \to \infty}}^{+} A_n(h) \Big\|_1
    \leq 
    \Big\| {\limsup_{n \to \infty}}^{+} A_n(h) \Big\|_1.
  \]
  But, since $A_n(h) \to F(h)$ bilaterally almost uniformly by \ref{itm:MainTool.2},
  we can apply Proposition \ref{prp:BAUImpLimsup} to conclude that
  \begin{eqnarray*}
      &      & \Big\| {\limsup_{n,m \to \infty}}^{+} A_n \circ B_m (f) \Big\|_1 \\
      & \leq & \big\| F: L_1(\M) \to L_1(\M) \big\| \, \| h \|_1\\
      & \leq & (1 + \epsilon) \, \big\| F: L_1(\M) \to L_1(\M) \big\|
                              \, \big\| (B_m)_m: L_\Phi(\M) \to L_1(\M;\ell_\infty) \big\| \| f \|_\Phi.
  \end{eqnarray*}
  and since $\epsilon$ is arbitrary we can conclude.
\end{proof}

Now, we have all the ingredients to obtain Corollary \ref{cor:MultiMartingale}.

\begin{proof}[{\bf Proof (of Corollary \ref{cor:MultiMartingale}.)}]
  Since $B_m(f) =  \EE_m^{[2]}(f)$ is a martingale sequence, we can apply Theorem \ref{thm:MartingalesLpMax}.\ref{eq:MartingalesLpMax2} in order to obtain that the maximal operator associated to $B_m$ maps $L \log^2 L (\M)$ into $L_1(\M;\ell_\infty)$. Similarly, by Cuculescu's theorem we obtain that $\EE_n^{[1]}(f)$ converge b.a.u. for every $f \in L_1(\M)$. An application Theorem \ref{thm:MainTool} gives the result.
\end{proof}
%follows from a combination of \ref{thm:LpBoundErgodic} and the Theorem \ref{thm:MainTool}. The same techniques also yields the following corollary.
The following corollary gives a similar result for multiparametric ergodic means.
\begin{corollary}
  \label{cor:MultErgodic}
  Let $(\M^{[i]}, \tau_i)$ for $i \in \{1, 2\}$ be two finite von Neumann algebras
  and $T_i: \M^{[i]} \to \M^{[i]}$ be Markovian operators
  and $F_i$ the projections onto their fixed subspaces $\ker(\Id - T_i)$.
  \begin{enumerate}[leftmargin=1.2cm, ref={\rm (\roman*)}, label={\rm (\roman*)}]
    \item \label{itm:MultErgodic.1}
    Let $(\M, \tau) = (\M^{[1]} \weaktensor \M^{[2]}, \tau_1 \otimes \tau_2)$,
    it holds that
    \[
      \Big\| \, {\limsup_{n, m \to \infty}}^{+} \big(M_n(T_1) \otimes M_m(T_2)\big)(f) \Big\|_1
      \, \lesssim \,
      \| f \|_{L \log^2 L(\M)}.
    \]
    As a consequence $(M_n(T_1) \otimes M_m(T_2))(f) \to (F_1 \otimes F_2)(f)$
    bilaterally almost uniformly as $n, m \to \infty$, for every $f \in L \log^2 L(\M)$.
    \item \label{itm:MultErgodic.2}
    Let $(\M, \tau) = (\M^{[1]} \ast \M^{[2]}, \tau_1 \ast \tau_2)$, it holds that
    \[
      \Big\| \, {\limsup_{n, m \to \infty}}^{+} \big(M_n(T_1) \ast M_m(T_2)\big)(f) \Big\|_1
      \, \lesssim \,
      \| f \|_{L \log^2 L(\M)}
    \]
     As a consequence $(M_n(T_1) \ast M_m(T_2)) \to (F_1 \ast F_2)(f)$ bilaterally
    almost uniformly as $n, m \to \infty$, for every $f \in L \log^2 L(\M)$.
    \item \label{itm:MultErgodic.3}
    Furthermore, if $T_2$ is symmetric, we can change $M_m(T_2)$ by $T_2^m$ in points
    \ref{itm:MultErgodic.1} and \ref{itm:MultErgodic.2} above.
  \end{enumerate}
\end{corollary}

\begin{proof}
  Again, the proof is an application of Theorem \ref{thm:MainTool}. We will only prove points \ref{itm:MultErgodic.1} and \ref{itm:MultErgodic.3} since \ref{itm:MultErgodic.2} follows like point \ref{itm:MultErgodic.1}.
  
  Notice that, by Theorem \ref{thm:LpBoundErgodic}, we have that the maximal associated to $\Id \otimes M_m(T_2)$ maps $L \log^2 L(\M)$ into $L_1(\M;\ell_\infty)$. It is a consequence of Yeadon's results that $M_n(T_1) \otimes \Id$ converge b.a.u. Gathering both results allow to apply Theorem \ref{thm:MainTool}.
  For \ref{itm:MultErgodic.3} we use that, when $T_2$ is symmetric, so is $T_2 \otimes \Id$ and therefore, by of \ref{thm:LpBoundErgodic}\ref{itm:LpBoundErgodic.4} the maximal operator given by powers of $\Id \otimes T_2$ is maximally bounded from $L \log^2 L(\M)$ into $L_1(\M;\ell_\infty)$.
\end{proof}

There are two natural generalization of the results above. One is in the case of asymmetric spaces, the other relates to higher iterations, i.e. working with $\EE_{n_1}^{[1]} \otimes \EE_{n_2}^{[2]} ... \otimes \EE_{n_r}^{[r]}$ instead of just $\EE_{n_1}^{[1]} \otimes \EE_{n_2}^{[2]}$. We will briefly cover both direction here.

The asymmetric result can be seen as a generalization of \cite{Litvinov2012Ergodic}, see also \cite{HongSun2018,Hong2016asymmetric}. We will need the following proposition, whose proof in immediate and this we omit it.

\begin{proposition}
  \label{prp:ColumnRowSquare}
  \
  \begin{enumerate}[leftmargin=1.2cm, label={\rm (\roman*)}, ref={\rm (\roman*)}]
    \item \label{itm:ColumnRowSquare.1}
    $\displaystyle{
      \big\| f^\ast f \big\|_{L_1 \log^s L}^\frac12
        \, \sim_{(s)} \, \| f \|_{L_2 \log^s L}
        \, \sim_{(s)} \, \big\| f \, f^\ast \big\|_{L_1 \log^s L}^\frac12}$
    \item \label{itm:ColumnRowSquare.2}
    $\displaystyle{
      \big\| (f^\ast_n f_n)_n \big\|_{L_1[\ell_\infty]}^\frac12 
        \, = \, \big\| (f_n)_n \big\|_{L_2[\ell_\infty^c]}
      \quad \mbox{ and } \quad
      \big\| (f_n \, f^\ast_n)_n \big\|_{L_1[\ell_\infty]}^\frac12
        \, = \, \big\| (f_n)_n \big\|_{L_2[\ell_\infty^r]}
    }$
    \item \label{itm:ColumnRowSquare.3}
    $\displaystyle{
      \big\| (f^\ast_n \, f_n)_n \big\|_{L_1^\loc[\ell_\infty/c_0]}^\frac12
        \, = \, \| (f_n)_n \|_{L_2^\loc[\ell_\infty^c/c_0^c]}
      \quad \mbox{ and } \quad
      \big\| (f_n \, f^\ast_n) \big\|_{L_1^\loc[\ell_\infty/c_0]}^\frac12
        \, = \, \big\| (f_n)_n \big\|_{L_2^\loc[\ell_\infty^r/c_0^r]}
    }$
  \end{enumerate}
\end{proposition}

Now, we will formulate the asymmetric analogue of Corollary \ref{cor:MultErgodic}.

\begin{corollary}
  \label{cor:Assymetric}
  Let $(\M^{[i]}, \tau_i)$ for $i \in \{1, 2\}$ be two finite von Neumann algebrasç
  and $(\M_m^{[i]})_m$ and $\EE_m^{[i]}: \M \to \M_m^{[i]}$ be as before.
  Let $(M, \tau)$ be $(\M^{[1]} \weaktensor \M^{[2]}, \tau_1 \otimes \tau_2)$.
  It holds that
    \[
      \label{eq:Assymetric.1}
      \big\| \big(\EE_n^{[1]} \otimes \EE_m^{[2]}\big)_{n,m}:
        L_2 \log^2 L(\M) \to L_2^\loc(\M;\ell_\infty^\dagger/c_0^\dagger) \big\|
      < \infty
    \]
    where $\dagger \in \{ r, c \}$. As a consequence $(\EE_n^{[1]} \otimes \EE_m^{[2]})(f) \to f$
    row and column almost uniformly as $n, m \to \infty$, for every $f \in L_2 \log^2 L(\M)$.
\end{corollary}

The same result follows with trivial changes for free products and in the case of ergodic means. 

\begin{proof}
  We are only going to prove the column case since both are identical.
  For every $f \in L_2 \log^2 L(\M)$, we have that $g = f^\ast \, f \in L \log^2 L(\M)$
  by Proposition \ref{prp:ColumnRowSquare} \ref{itm:ColumnRowSquare.1}. Thus, applying 
  $\EE_n^{[1]} \otimes \EE_m^{[2]}$ to $g$ and using Corollary \ref{cor:MultiMartingale}
  \ref{itm:MultiMartingale.1} gives that
  \[
    \Big\| \big(\EE_n^{[1]} \otimes \EE_m^{[2]}\big)(g) \Big\|_{L_1^\loc[\ell_\infty/c_0]}^\frac12
    \, \lesssim \,
    \| g \|_{L \log^2 L}^\frac12,
  \]
  but the right hand side is equivalent to the $L_2 \log^2 L$-norm of $f$, while, by Kadison's inequality,
  the left hand side satisfies that
  \[
    \big(\EE_n^{[1]} \otimes \EE_m^{[2]}\big)(g)
    =
    \big(\EE_n^{[1]} \otimes \EE_m^{[2]}\big)(f^\ast f)
    \geq
    \big(\EE_n^{[1]} \otimes \EE_m^{[2]}\big)(f)^\ast \big(\EE_n^{[1]} \otimes \EE_m^{[2]}\big)(f)
  \]
  Applying Proposition \ref{prp:ColumnRowSquare} \ref{itm:ColumnRowSquare.3}
  gives the desired result.
\end{proof}

In order to extend Corollaries \ref{cor:MultErgodic} and \ref{cor:MultiMartingale} to the case of larger index sets we need the following proposition, which is a straightforward generalization of \cite[p. 1121]{HongSun2018}. Observe also the the proposition bellow is a generalization of
\cite[Theorem 2.3]{Hu2009} or Theorem \ref{thm:Extrapolation}.

\begin{proposition}
  \label{prp:multilogarithms}
  Let  $(S_m)_m$ be a
  family of Markovian operators such that
  \[
    \big\| (S_m)_m: L_p(\M) \to L_p(\M;\ell_\infty) \big\|
      \, \lesssim \, \max \Big\{ 1, \Big( \frac1{p - 1} \Big)^\alpha \Big\} 
  \]
  for some $\alpha \geq 0$. There, for every $\beta \geq 0$ it holds that
  \[
    \big\| (S_m)_m: L \log^{\beta + \alpha} L(\M) \to L \log^{\beta} L (\M;\ell_\infty) \big\|
    \, \lesssim \, 
    \infty.
  \]
\end{proposition}

The proof uses as a key ingredient the atomic decomposition of noncommutative $L \log^s L$ spaces obtained in \cite[Appendix]{HongSun2018} as direct generalization of \cite{Tao2001Extrapolation}. Indeed, every element of the unit ball of $L \log^s L$ can be expressed as a combination of \emph{atoms}
\[
  f = \sum_{k = 0}^\infty \lambda_k a_k
\]
where $\lambda_k \geq 0$ satisfy that $\sum_{k = 1} \lambda_k \leq 1$. An element $a \in \M$ is an 
atom for $L \log^s L$ iff it is supported in projection $e \in \Prj(\M)$ with $\tau(e) = 2^k$ for some
$k \geq 0$ and
\[
  \| a \|_\infty \leq \frac1{\tau(e)} \, \Big( 1 + \log_+ \Big( \frac1{\tau(e)} \Big) \Big)^{-s}.
\]
The atomic decomposition also gives an equivalence in norm
\[
  \| f \|_{L \log^s L}
  \, \sim \, 
  \inf \left\{ \sum_{k = 0}^\infty \lambda_k
              : f = \sum_{k=0}^\infty \lambda_k a_k \quad
                \mbox{ for } a_k \mbox{ atom }
       \right\}
\]
\begin{proof}
  Apart from the atomic decomposition, the other ingredient is the fact that, for finite algebras $\M$ we have
  \[
    \| f \|_{L \log^s L} \lesssim_{(s)} \frac1{(p - 1)^s} \| f \|_p,
  \]
  for every $1 < p \leq \infty$, see \cite{Bennett1988}. Start by noticing that, by Minkowsky's inequality, it is enough to prove the result for atoms. Let us fix $a \in L \log^{\alpha + \beta} L(\M)$ an atom. We have that
  \begin{eqnarray}
    \big\| \big( S_m(a) \big)_m \big\|_{L \log^\beta L[\ell_\infty]}
      & \lesssim_{(s)} & \frac1{(p - 1)^\beta} \, \big\| \big( S_m(a) \big)_m \big\|_{L_p[\ell_\infty]} \nonumber\\
      & \lesssim & \frac1{(p - 1)^{\alpha + \beta}} \, \big\| a \big\|_{p} \label{eq:ProofAtomic.2}\\
      & \leq & \frac1{(p - 1)^{\alpha + \beta}} \, \tau(e)^\frac1{p} \, \frac1{\tau(e)} \Big( 1 + \log_+ \frac1{\tau(e)} \Big)^{-(\alpha + \beta)} \label{eq:ProofAtomic.3}
  \end{eqnarray}
  for every $1 < p \leq \infty$, we have used the hypothesis on \eqref{eq:ProofAtomic.2}
  and the conditions on the atom  in \eqref{eq:ProofAtomic.3}. We can now fix $p$ depending on the atom to be
  \[
    p - 1 =  \Big( 1 + \log_+ \Big( \frac1{\tau(e)} \Big) \Big)^{-1}.
  \]
  That gives
  \[
    \big\| \big( S_m(a) \big) \big\|_{L \log^\beta L[\ell_\infty]}
    \, \leq \,
    \tau(e)^{(1-p)/p}
  \]
  and we conclude noticing that the last expression is uniformly bounded by the choice of $p$ relative to $\tau(e)$. 
\end{proof}

As a corollary, we obtain the following

\begin{corollary}
  \label{cor:MultilogThreeEx}
  Let $(\M,\tau)$ be a finite von Neumann algebra, $T$ a Markovian operator
  and $\EE_n:\M \to \M_m \subset  \M$ the conditional
  expectations associated to a filtration $(\M_m)$, we have that, for every
  $\beta \geq 0$
  \begin{enumerate}[leftmargin=1.2cm, ref={\rm (\roman*)}, label={\rm (\roman*)}]
    \item \label{itm:multilogarithms.1} 
    $\displaystyle{
      \big\| (\EE_n)_n: L \log^{\beta + 2} L(\M) \to L \log^{\beta} L (\M;\ell_\infty) \big\|
      \, < \, 
      \infty.
    }$
    \item \label{itm:multilogarithms.2}
    $\displaystyle{
      \big\| (M_n(T))_n: L \log^{\beta + 2} L(\M) \to L \log^{\beta} L (\M;\ell_\infty) \big\|
      \, < \, 
      \infty.
    }$
    \item \label{itm:multilogarithms.3} 
    $\displaystyle{
      \big\| (T^n)_n: L \log^{\beta + 2} L(\M) \to L \log^{\beta} L (\M;\ell_\infty) \big\|
      \, < \, 
      \infty,
    }$
    whenever $T$ is symmetric.
  \end{enumerate}
\end{corollary}

Now, we have all the ingredients to formulate our result in the multiparametric case.
In particular, we obtain that

\begin{corollary}
  \label{cor:MartingalesMultiIndex}
  Let $\M^{[i]}$ be a collection of von Neumann algebras, for $i \in \{1, 2, ..., d\}$,
  and assume that there are filtrations $\M_m^{[i]} \subset \M^{[i]}$ approximating
  each algebra and denote by $\EE_m^{[i]}$ their conditional expectations. 
  Let $(\M, \tau) = (\M^{[1]} \weaktensor \M^{[2]} \weaktensor ... \weaktensor\M^{[d]},
  \tau_1 \otimes \tau_2 ... \, \otimes \tau_d)$. It holds that
  \[
    \Big\| \, {\limsup_{n_1 \, ... \, n_d \to \infty}}^{+}
              \big(\EE_{n_1}^{[1]} \otimes \EE_{n_2}^{[2]} ... \, \otimes \EE_{n_d}^{[d]} \big)(f)
    \Big\|_1
    \, \lesssim \,
    \| f \|_{L \log^{2 \, (d - 1)} L(\M)}.
  \]
  As a consequence $\EE_{n_1}^{[1]} \otimes \EE_{n_2}^{[2]}  ... \, \otimes \EE_{n_d}^{[d]}(f) \to f$
  bilaterally almost uniformly as $n_1, n_2, ... n_d \to \infty$, for every $f \in L \log^{2 \, (d - 1)} L(\M)$.
  The same result holds after changing the tensor products by free products
  as in Corollary \ref{cor:MultErgodic}.
\end{corollary}

\begin{proof}
The proof is just an iteration of the Theorem \ref{thm:MainTool} using the positivity of the conditional expectations.
Indeed, starting with an element $f \in L \log^{2 \, (d - 1)} L$ gives, by Proposition \ref{cor:MultilogThreeEx} \ref{itm:multilogarithms.1} an element $f_1 \in L \log^{2 \, (d - 1)} L$ such that
\[
  \Id \otimes \EE_{n_d}^{[d]}(f) \leq f_1
\]
and its $L \log^{2 \, (d - 2)} L$ norm is controlled by the $L \log^{2 \, (d - 2)} L$ norm of $f$. Applying this process $d - 1$ times gives an element $f_{d-1} \in L_1$ and we can apply the bilateral almost uniform convergence of
$\EE_{n_1}^{[1]} \otimes \Id(f)$ for every $f \in L_1$ to obtain a bound of the $\limsup^+$ after Proposition \ref{prp:BAUImpLimsup}.
\end{proof}

The same result holds for ergodic means giving the following 

\begin{corollary}
  \label{cor:ErgodicMultiIndex}
  Let $\M^{[i]}$ be a collection of finite von Neumann algebras, for $i \in \{1, 2, ..., d\}$,
  and $T_i: \M^{[i]} \to \M^{[i]}$ be Markovian operators, fix
  $(M, \tau) = (\M^{[1]} \weaktensor \M^{[2]} \weaktensor \cdots \M^{[d]},
  \tau_1 \otimes \tau_2 ... \, \tau_d)$, it holds that
  \[
    \Big\| \, {\limsup_{n_1 \, ... \, n_d \to \infty}}^{+}
              \big(M_{n_1}(T_1) \otimes M_{n_2}(T_2) ... \, \otimes M_{n_d}(T_d) \big)(f)
    \Big\|_1
    \, \lesssim \,
    \| f \|_{L \log^{2 \, (d - 1)} L(\M)}
  \]
  As a consequence $M_{n_1}(T_1) \otimes M_{n_2}(T_2) ... \otimes M_{n_d}(T_d)(f)
  \to \big( F_1 \otimes F_2 ... \otimes F_d \big)f$
  
  Furthermore, if $T_2$, $T_3$, ..., $T_d$ are symmetric, we have that
  \[
    \Big\| \, {\limsup_{n_1 \, ... \, n_d \to \infty}}^{+}
              \big(M_{n_1}(T_1) \otimes T_2^{n_2} ... \, \otimes T_d^{n_d} \big)(f)
    \Big\|_1
    \, \lesssim \,
    \| f \|_{L \log^{2 \, (d - 1)} L(\M)}.
  \]
  The same result holds after changing the tensor products by free products
\end{corollary}

\section{Almost uniform converge in $\FF_2$ \label{sec:F2}}
We will also denote by $\CC[\FF_2] \subset \L \FF_2$ the free group algebra given by finite combinations of elements $\lambda_\omega$ in $\L \FF_2$. Clearly $\CC[\FF_2]$ is weak-$\ast$ dense in $\L \FF_2$ and norm dense in every $L_p(\L \FF_2)$ for $1 \leq p < \infty$ as well as in the Orlicz spaces $L \log^\alpha L (\L \FF_2)$. Again, it is immediate that $\| S_t(f) - f \|_\infty \to 0$ as $t \to 0$ for every $f \in \CC[\FF_2]$. As stated in the introduction, see Problem \ref{prb:ConvergenceFreeIntro}, the boundedness of the maximal operator associated to $(S_t)_t$ will give bilateral almost uniform convergence to the initial data as $t \to 0$. Sadly, that result is still unknown in $L_1$. Here, we will improve the best known result by giving a class of operators $\C$ between $L \log^2 L(\L \FF_2)$ and $L_1(\L \FF_2)$ for which bilateral almost uniform convergence holds. 

Let $\langle\langle a \rangle\rangle$ and $\langle\langle b \rangle\rangle$ be the normal subgroups generated by $a$ and $b$ respectively. Denote its associated quotient maps by $n_a: \FF_2 \to \FF_2 / \langle\langle b \rangle\rangle \cong \ZZ$ and $n_b: \FF_2 \to \FF_2 / \langle\langle a \rangle\rangle \cong \ZZ$. For some reduced word $\omega = s_1^{n_1} s_2^{n_2} ... s_\ell^{n_\ell}$, where each $s_i \in \{a,b\}$, the quotient maps are given by
\begin{eqnarray*}
  n_a \big( s_1^{n_1} s_2^{n_2} ... s_\ell^{n_\ell} \big) & = & \sum_{i | s_i = a} n_i,\\
  n_b \big( s_1^{n_1} s_2^{n_2} ... s_\ell^{n_\ell} \big) & = & \sum_{i | s_i = b} n_i.
\end{eqnarray*}
Intuitively $n_a$ and $n_b$ count the number of $a$'s and the number of $b$'s with multiplicity. 
Let $\Sigma \subset \FF_2$ be the subset of reduced words satisfying that there are no sign changes in the $a$ or $b$. I.e. for every word $\omega$, all of their $a$'s have positive (resp. negative) exponent and the same holds for the $b$'s. The key observation that will allow us to apply transference techniques is that, for every $\omega \in \Sigma$, we do have that
\[
  | \omega | = | n_a(\omega) | + | n_b(\omega) |.
\]
Let $\Theta: \CC[\FF_2] \to \CC[\FF_2] \algtensor \CC[\ZZ] \algtensor \CC[\ZZ]$ be the map given by linear extension of $\lambda_\omega \mapsto \lambda_\omega \otimes \lambda_{n_a(\omega)} \otimes \lambda_{n_b(\omega)}$. We have the following.
\begin{proposition}
  \label{prp:ExtensionTheta}
  The map $\Theta$ extends to a normal and faithful and trace preserving $\ast$-homomorphism
  \[
    \Theta: \L \FF_2 \longrightarrow \L \FF_2 \weaktensor \L(\ZZ^2) \cong \L \FF_2 \weaktensor L_\infty(\TT^2).
  \]
\end{proposition} 
The proof is a routine application of Fell's absorption principle. Indeed, after
identifying $\ell_2(\ZZ^2)$ with $\ell_2(\ZZ) \otimes_2 \ell_2(\ZZ)$,
the map $V: \ell_2(\FF_2) \otimes_2 \ell_2(\ZZ^2) \to \ell_2(\FF_2) \otimes_2 \ell_2(\ZZ^2)$ given by linear extension of
\[
  \delta_\omega \otimes \delta_{k_1} \otimes \delta_{k_2}
  \longmapsto
  \delta_\omega \otimes \delta_{k_1 - n_a(\omega)} \otimes \delta_{k_2 - n_b(\omega)}
\]
is a unitary and satisfies that the following diagram commutes.
\begin{equation*}
  \label{dia:FellAbsorption}
  \xymatrix@C=7em{
    \ell_2(\FF_2) \otimes_2 \ell_2(\ZZ^2) \ar[r]^{V} \ar[d]^{\lambda_\omega \otimes \lambda_{n_a(\omega)} \otimes \lambda_{n_a(\omega)}} & \ell_2(\FF_2) \otimes_2 \ell_2(\ZZ^2) \ar[d]^{\lambda_\omega \otimes \1 \otimes \1}\\
    \ell_2(\FF_2) \otimes_2 \ell_2(\ZZ^2) \ar[r]^{V} & \ell_2(\FF_2) \otimes_2 \ell_2(\ZZ^2)
  }
\end{equation*}
The result follows after identifying $\L \ZZ^2$ with $L_\infty(\TT^2)$. The fact that $\Theta$ is trace preserving, where the natural trace of $\L \FF_2 \otimes L_\infty(\TT^2)$ is given by the tensor product of $\tau$ and integration against the Haar measure of $\TT^2$, is trivial. Therefore $\Theta$ is faithful and its image is isomorphic to $\L \FF_2$. Given any $\bar\eta \in \TT^2$, we will denote by $\Theta_{\bar\eta}$ the trace preserving automorphism of $\L \FF_2$ given by 
\[
  \Theta_{\bar\eta}(f) = \mathrm{ev}_{\bar\eta} \circ \Theta(f),
\]
where $\mathrm{ev}_{\bar\theta}$ denotes the evaluation in $\bar\theta \in \TT^2$. Alternatively, one can see $\Theta_{\bar\theta}: \L \FF_2 \to \L \FF_2$ as the map given by
\[
  \Theta_{\bar\theta} \Big( \sum_{\omega \in \FF_2} \widehat{f}(\omega) \lambda_\omega \Big)
  = 
  \sum_{\omega \in \FF_2}
    e^{2 \pi i \big( n_a(\omega) \theta_1 + n_b(\omega) \theta_2 \big)} \widehat{f}(\omega) \lambda_\omega.
\]
We shall denote by $\L \FF_2 {|}_{\Sigma} \subset \L \FF_2$ the subset of elements supported in $\Sigma$, i.e:
\[
  \L \FF_2 {|}_{\Sigma}
  =
  \Big\{ f \in \L \FF_2 : f = \sum_{\omega \in \Sigma \subset \FF_2} \widehat{f}(\omega) \, \lambda_\omega  \Big\}.
\]
Similarly, we will denote by $L_p(\L \FF_2){|}_\Sigma$ or $L \log^\alpha L(\L \FF_2){|}_\Sigma$ the subsets of elements $f \in L_p(\L \FF_2)$ or $L \log^\alpha L(\L \FF_2)$ such that $\tau(f \lambda_g^\ast) = 0$ for every $g \not\in \Sigma$. Let $P_s:L_\infty(\TT^d) \to L_\infty(\TT^d)$  be the Poisson semigroup of the torus, given by
\[
  P_s(f) = 
  P_s \Big( \sum_{k \in \ZZ^d} \widehat{f}(k) e^{2 \pi i k \theta} \Big)
  = 
  \sum_{k \in \ZZ^d} \widehat{f}(k) e^{-s|k|} e^{2 \pi i k \theta}
\]
We have that the following diagram commutes
\begin{equation*}
  \label{dia:SigmaLinft}
  \xymatrix@C=7em{
    \L \FF_2{|}_\Sigma \ar[d]^{S_t} \ar[r]^-{\subset} & \L \FF_2 \ar[r]^-{\Theta} & \L \FF_2 \weaktensor L_\infty(\TT^2) \ar[d]^{\Id \otimes P_t} \\
    \L \FF_2{|}_\Sigma \ar[r]^-{\subset} & \L \FF_2 \ar[r]^-{\Theta} & \L \FF_2 \weaktensor L_\infty(\TT^2)
  }
\end{equation*}
To extend the commutativity of the diagram above to the $L_1$. We need the following proposition.

\begin{proposition}
  \label{prp:IsometricTheta}
  Let $\Theta: \L \FF_2 \longrightarrow \L \FF_2 \weaktensor L_\infty(\TT^2)$ be as before. We have that
  \begin{enumerate}[leftmargin=1.5cm, ref={\rm (\roman*)}, label={\rm (\roman*)}]
    \item \label{itm:IsometricTheta.1}
    $\Theta: L_1(\L \FF_2) \longrightarrow L_1(\L \FF_2 \otimes L_\infty(\TT^2))$
    extends to a positivity-preserving isometry.
    \item \label{itm:IsometricTheta.2}
    $\Theta: L_1^\loc(\L \FF_2; \ell_\infty/c_0) \longrightarrow L_1(\L \FF_2 \otimes L_\infty(\TT^2); \ell_\infty/c_0)$
    extends to a order-preserving quasi-isometry of the positive cones.
  \end{enumerate}
\end{proposition}

\begin{proof}
  The first point is a trivial consequence of the fact that $\Theta$ is trace-preserving.
  In order to prove \ref{itm:IsometricTheta.2} we have to show that the following two
  inequalities hold for every $f \in L_1(\M)_+$
  \begin{eqnarray}
    \big\| (\Theta(f_n))_n \big\|_{L_1^\loc[\ell_\infty/c_0]}
      & \leq & \big\| (f_n)_n \big\|_{L_1^\loc[\ell_\infty/c_0]} \label{eq:ProofIso.1}\\
    \big\| (f_n)_n \big\|_{L_1^\loc[\ell_\infty/c_0]}
      & \leq & 2 \big\| (\Theta(f_n))_n \big\|_{L_1^\loc[\ell_\infty/c_0]} \label{eq:ProofIso.2}.
  \end{eqnarray}
  For \eqref{eq:ProofIso.1}, start by fixing $\epsilon > 0$. We have that there is a $p \in \Prj(\L \FF_2)$, $N$ an integer and $g \in L_1(\L \FF_2)$ such that $p^\perp \, f_n \, p^\perp \leq g$ for every $N < n$ and $\| g \|_1 \leq (1 + \delta) \| (f_n)_n \|_{L_1^\loc[\ell_\infty/c_0]}$ for every $\delta$. This implies that for $\epsilon > 0$ there is a projection $q = \Theta(p)$ with $\tau(q) < \epsilon$ and $h \in L_1(\L \FF_2 \weaktensor L_\infty(\TT^2))$ given by $h = \Theta(g)$ satisfying that $q^\perp \, \Theta(f_n) \, q^\perp \leq h$ for every $N < n$, by point \ref{itm:IsometricTheta.2} we have that $\| h \|_1 \leq (1 + \delta) \| (f_n)_n \|_{L_1^\loc[\ell_\infty/c_0]}$. Since $\delta$ is arbitrarily small, we can conclude.
  For \eqref{eq:ProofIso.2} we are going to proof that
  \[
    \big\| (\Theta(f_n))_n \big\|_{L_1^\loc[\ell_\infty/c_0]} \leq 1 
    \quad \mbox{ implies that } \quad
    \big\| (f_n)_n \big\|_{L_1^\loc[\ell_\infty/c_0]} \leq 2.
  \]
  By the first assumption and the definition of the $L_1^\loc[\ell_\infty/c_0]$-norm we have that for every $\epsilon$ there is a projection $q \in \Prj(\L \FF_2 \weaktensor L_\infty(\TT^2))$, with $\varphi(q) < \epsilon$, an integer $N$ and $g \in L_1(\L \FF_2 \weaktensor L_\infty(\TT^2))$ with $\| g \|_1 \leq 1 + \delta$
  such that $q^\perp \, \Theta(f_n) \, q^\perp \leq g$ for $n N$. In order to reduce the identity to $\L \FF_2$ notice that for almost every $\bar\eta \in \TT^2$, we have that
  \[
    q_{\bar\eta}^\perp \, \Theta_{\bar\eta}(f) \, q^\perp_{\bar\eta} \leq g_{\bar\eta}
  \]
  holds, where $q_{\bar\eta} = (\mathrm{ev}_{\bar\eta} \otimes \Id)(q)$ and the same for $g$.
  Observe that we have the following inequalities as a consenquence of Chebichyev's inequality
  \begin{eqnarray*}
    \big| \big\{ \bar\theta \in \TT^2 : \tau(p_{\bar\theta}) & < & 2 \epsilon \big\} \big| \geq \frac12,\\
    \big| \big\{ \bar\theta \in \TT^2 : \| g_{\bar\theta} \|_1 & < & (2 + \delta) \| g \|_1 \big\} \big| > \frac12,
  \end{eqnarray*}
  Therefore, there is a $\bar\theta$ that lays in the intersection of both sets. That $\bar\theta$ satisfies that
  \begin{itemize}
    \item $\tau(q_{\bar\theta}) < 2 \epsilon$
    \item $q_{\bar\theta}^\perp \, \Theta_{\bar\theta}(f_n) \, q^\perp_{\bar\theta} \leq g_{\bar\theta}$ for every $N < n$.
    \item $\| g_{\bar\theta} \|_1 \leq (2 + \delta)$.
  \end{itemize}
  Applying $\Theta_{- \bar\theta}$ to $q_{\bar\theta}$, $g_{\bar\theta}$ and using the fact that the $\sup$ in the formula in Proposition \ref{prp:AltNorm} is monotone, and thus it can be exchanged by a limit $\epsilon \to 0$ gives that the $L_1^\loc[\ell_\infty/c_0]$-norm is smallest that $2 + \delta$. But, since $\delta$ can be taken arbitrarily small we can conclude.
\end{proof}

Now, we can proceed to prove the main theorem of this section.

\begin{proof}[{\bf Proof (of Theorem \ref{thm:FreegroupExt})}]
  Fix $t_n$. Observe that we have the following commutative diagram. 
  \begin{equation*}
    \xymatrix@C=5em@R=3em{
      L_1\big(\L \FF_2\big){|}_\Sigma \ar[d]^{(S_{t_n})_n} \ar[r]^{\subset}
        & L_1\big(\L \FF_2\big) \ar[r]^{\Theta} & L_1(\L \FF_2 \weaktensor L_\infty(\TT^2))
          \ar[d]^-{(\Id \otimes P_{t_n})_n} \\
      L_1^\loc \big( \L \FF_2; \ell_\infty/c_0 \big) \ar[r]^{=}
        & L_1^\loc \big(\L \FF_2; \ell_\infty/c_0 \big) \ar[r]^-{\Theta} & L_1^\loc \big(\L \FF_2 \weaktensor L_\infty(\TT^2); \ell_\infty/c_0\big)
    }
  \end{equation*}
  We have that $(\Id \otimes P_s)$ is a subordinated semigroup. Therefore, by Proposition \ref{prp:subbordinationBAU}, $(Id \otimes P_s)(f) \to f$ bilaterally almost uniformly as $s \to 0^+$. By Theorem \ref{prp:BAUImpLimsup} we have that 
  \[
    \big\| (\Id \otimes P_{s_n})
      :L_1(\L \FF_2 \weaktensor L_\infty(\TT^2)) \longrightarrow L_1^\loc(\L \FF_2 \weaktensor L_\infty(\TT^2); \ell_\infty/c_0) \big\| \leq 1.
  \]
  But now, by Proposition \ref{prp:ExtensionTheta}, we have that the horizontal arrows in the commutative diagram are quasi isometric. That implies \eqref{eq:FreegroupExt} holds. The result follows.
\end{proof}

\subsection*{The two-parametric problem in the free group \label{ssc:two-parameterF2}}
There is a natural two-parametric generalization of Problem \ref{prb:ConvergenceFreeIntro} for the free group $\FF_2$.
That is to determine what is the largest class $\D \subset L_1(\L \FF_2)$ for which the two-parametric symmetric Markovian semigroup given by $(s,t) \mapsto P_s \ast P_t$ satisfies that
\[
  \big( P_s \ast P_t \big)(f) \to f
  \quad \mbox{ bilaterally almost uniformly, as } s, t \to 0^+,
\]
for every $f \in \D$. A straightforward adaptation of Corollary \ref{cor:MultErgodic} gives the following

\begin{corollary}
  \label{cor:twoParameterFree}
  For every $f \in L \log^2 L (\L \FF_2)$ ans sequences $s_n, t_n \to 0$, we have that 
  \begin{equation*}
    \label{eq:twoParameterFree}
    \Big\| \, {\limsup_{n, m \to \infty}}^+ P_{s_n} \ast P_{t_m} (f) \Big\|_1
    \, \lesssim \,
    \| f \|_{L \log^2 L} 
  \end{equation*}
  As a consequence $P_s \ast P_t(f) \to f$ bilaterally almost uniformly as $s,t \to 0$,
  for every $f \in L \log^2 L(\L \FF_2)$.
\end{corollary}

It is natural to conjecture that, following the same techniques employed to prove Theorem \ref{thm:FreegroupExt}, there would be two-parametric convergence for $P_s \ast P_t$ whenever $f$ lays in the class $\D$
\begin{equation}
  \label{eq:ClassDdef}
  \D = L \log^2 L(\L \FF_2) + L \log L(\L \FF_2) {|}_\Sigma.
\end{equation}
Nevertheless, this result is conditioned by the problem of determining whether the following holds
\begin{problem}
  \label{prb:CBMaximal}
  Let $(P_s)_{s \geq 0}$ be the usual Poisson semigroup on $\TT^d$. Does it hold that
  \[
    \big\| (P_s)_s: \, L_p(\TT^d) \to L_p(\TT^d; \ell_\infty) \big\|_\cb
    \, \lesssim \, 
    \max \Big\{ 1, \frac1{p - 1} \Big\}.
  \]
\end{problem}
Recall that, by \eqref{eq:CBSchattenNorm} and \eqref{eq:SchattenMaximal}, we have that the complete norm above is given by
\[
  \begin{split}
    \big\| (P_s)_s: & L_p(\TT^d) \to L_p(\TT^d; \ell_\infty) \big\|_\cb\\
    & = \sup_{m \geq 1} \Big\{ \big\| (\Id \otimes P_s)_s: S_p^m[L_p(\TT^d)]
                                      \to S_p^m[L_p(\TT^d; \ell_\infty)
                               \big\|
                        \Big\}\\
    & = \sup_{m \geq 1} \Big\{ \big\| (\Id \otimes P_s)_s: L_p(M_m \otimes L_\infty(\TT^d))
                                      \to L_p(M_m \otimes L_\infty(\TT^d); \ell_\infty)
                               \big\|
                        \Big\}.
  \end{split}
\]
It is also easily obtained that the complete norm in $L_p$ of the maximal operator associated with the Poisson semigroup on $\TT^d$ has a norm that is bounded by $(p -1)^{-2}$. But, as far as the knowledge of the authors go, it is not known whether this bound can be lowered to $(p-1)^{-1}$, which is the classical non-complete norm. If Problem \ref{prb:CBMaximal} were to have a positive solution, that would imply, by a routine application of Theorem \ref{thm:MainTool}, that the rightmost vertical arrow in the commutative diagram below is bounded
\begin{equation*}
  \xymatrix@C=4em@R=3em{
    L \log L \big(\L \FF_2\big){|}_\Sigma \ar[d]^{(P_{t_n} \ast P_{s_m})_{n,m}} \ar[r]^{\subset}
      & L \log L \big(\L \FF_2\big) \ar[r]^-{\Theta} & L \log L(\L \FF_2 \weaktensor L_\infty(\TT^2))
      \ar[d]^-{(\Id \otimes P_{t_n} \otimes P_{s_m})_{n, m}} \\
    L_1^\loc \big( \L \FF_2; \ell_\infty/c_\Delta \big) \ar[r]^{=}
      & L_1^\loc \big(\L \FF_2; \ell_\infty/c_\Delta \big) \ar[r]^-{\Theta}
      & L_1^\loc \big(\L \FF_2 \weaktensor L_\infty(\TT^2); \ell_\infty/c_\Delta\big)
  }
\end{equation*}
Using that the horizontal arrows are isometric would give that $(\Id \otimes P_s)$ induces a maximally bounded map from $L \log L(\L \FF_2){|}_\Sigma$ to $L_1^\loc(\L \FF_2; \ell_\infty/c_\Delta)$. Therefore, there would be bilateral almost uniform convergence for functions in the class $\D$ of \eqref{eq:ClassDdef} for the two-parametric semigroup $P_s \ast P_t$.

\section{\bf The Strong Maximal inequality \label{sct:StrongIneq}}

In this section we will work under the assumption that $(\M, \tau) = (\M^{[1]} \weaktensor \M^{[2]}, \tau_1 \otimes \tau_2)$, where $(\M^{[i]} , \tau_i)$ is a noncommutative probability space, $i \in \{1,2\}$. We also fix  filtrations $(\M_m^{[i]})$  of $\M_i$ with associated conditional expectations $\EE_m^{[i]}: \M^{[i]} \to \M^{[i]}_m$. We will denote $\EE_n^{[1]} \otimes \EE_m^{[2]}(f)$ by $f_{n,m}$ at times.
Our goal will be to prove Theorem \ref{thm:StrongIneq}, which gives an $\varepsilon$-perturbation of the weak type $(\Phi,\Phi)$ for the strong maximal in two variables, where $\Phi(t) = t \, (1 + \log t)^2$. We also suspect that a similar inequality holds for multiparametric ergodic means using Yeadon's inequality as a tool instead of Cuculescu's inequality.

\subsection*{\bf An improvement of Cuculescu's inequality} 
Let $(\N, \tau)$ be a finite von Neumann algebra. Given a positive $f \in L_1(\N)$, Cuculescu's inequality gives an explicit combinatorial description of the noncommutative maximal weak type $(1,1)$ inequality for $f \mapsto (\EE_n(f))_n$.

\begin{theorem}[{\bf Cuculescu's projections \cite{Cuculescu1971}}]
  \label{thm:CuculescuConstruction}
  Let $(\N_n)_n$ be a filtration over $\N$ and $(\EE_n)_n$ the associated conditional expectations.
  Given $f \in L_1(\N)_+$ and $\lambda > 0$, the family of decreasing projections
  $q_1(\lambda) \geq q_2(\lambda) \geq \ldots q(\lambda)$ given by the recursive formula
  \begin{eqnarray*}
    q_1(\lambda) & = & \1,\\ 
    q_n(\lambda)
      & = & q_{n-1}(\lambda) \, \1_{[0,\lambda]} \big( q_{n-1}(\lambda) \, f_n \, q_{n-1}(\lambda) \big)\\
      & = & \1_{[0,\lambda]} \big( q_{n-1}(\lambda) \, f_n \, q_{n-1}(\lambda) \big) \, q_{n-1}(\lambda),\\
    q(\lambda)
      & = & \bigwedge_{n \geq 1} q_n(\lambda),
  \end{eqnarray*}
  satisfies
  \begin{enumerate}[leftmargin=1.2cm, label={\rm (\roman*)}, ref={\rm (\roman*)}]
    \item $q_n(\lambda) \in \N_n$.
    \item $q_n(\lambda) \, f_n \, q_n(\lambda) \leq \lambda q_n(\lambda)$.
    \item $q_n(\lambda)$ commutes with $q_{n-1}(\lambda) \, f_n \, q_{n-1}(\lambda)$.
    \item Furthermore, 
    \begin{equation}  
      \label{eq:OriginalCuculescu}
      \tau \big( q(\lambda)^\perp \big) \leq \frac{1}{\lambda} \| f \|_1.
    \end{equation}
  \end{enumerate}
\end{theorem}

The following lemma gives a refinement of inequality \eqref{eq:OriginalCuculescu} in which the right hand side is further restricted to the spectral region over which $f$ is large with respect to $\lambda$.

\begin{lemma}
  \label{lem:NewCuculescu}
  Let $f \in L_1(\N)_+$, $\lambda \geq 0$ and $q(\lambda)$ be like in the Theorem \ref{thm:CuculescuConstruction}.
  We also have that
  \[
    \tau \big( q(\lambda)^\perp \big)
    \, \leq \,
    \frac{2}{\lambda} \tau \big( f \, \1_{(\frac{\lambda}{2},\infty)}(f) \big).
  \]
\end{lemma}

\begin{proof}
  Let $p_n(\lambda) = q_{n-1}(\lambda) - q_n(\lambda)$ for $n > 1$ and observe that 
  \begin{equation}
    \label{eq:PrelinNewCuculescu}
    p_n(\lambda) \, \EE_n \big( f \1_{(\frac{\lambda}{2},\infty)}(f) \big) \, p_n(\lambda)
    \, \geq \,
    \frac{\lambda}{2} p_n(\lambda).
  \end{equation}
  Indeed, $p_n(\lambda) f_n p_n(\lambda) = p_n(\lambda) q_{n-1}(\lambda) f_n q_{n-1}(\lambda) p_n(\lambda) \geq \lambda p_n(\lambda)$ and 
  \[
    \begin{split}
      p_n(\lambda) \, & \EE_n \big( f \1_{(\frac{\lambda}{2},\infty)}(f) \big) \, p_n(\lambda) \\
      & =  p_n(\lambda) \, f_n \, p_n(\lambda) - p_n(\lambda) \, \EE_n \big( f \, \1_{[0,\frac{\lambda}{2}]}(f) \big) \, p_n(\lambda) \, \geq \, \frac{\lambda}{2} \, p_n(\lambda).
     \end{split}
  \]
  Next, we use inequality \eqref{eq:PrelinNewCuculescu} to obtain
  \begin{eqnarray*}
    \tau \big( \1 - q(\lambda) \big) & = &  \sum_{n > 1} \tau \big( p_n(\lambda) \big) \\
      & \leq & \frac{2}{\lambda} \, \sum_{n > 1} \tau \Big( p_n(\lambda) \, \EE_n \big( f \, \1_{(\frac{\lambda}{2}, \infty)}(f) \big) \, p_n(\lambda) \Big) \, \leq \, \frac{2}{\lambda} \tau \big( f \1_{(\frac{\lambda}{2},\infty)}(f) \big)
  \end{eqnarray*}
  which proves the claimed assertion.
\end{proof}

\begin{remark}
  \label{rmk:SteinLlogL}
  In the commutative case, the formula in Lemma \ref{lem:NewCuculescu} is very well known. For instance, a two sided version of the inequality was used in \cite{Stein1969LlogL} to prove that the $L_1$-norm of the Hardy-Littlewood maximal function associated to $f$ is comparable to the $L \log L$-norm of $f$. This can be obtained immediately integrating the formula in Lemma \ref{lem:NewCuculescu} with respect to $\lambda$ and applying Fubini's theorem.
  
  That computation gives an interesting consequence in the noncommutative case. Recall that it was shown in \cite{Hu2009} that the noncommutative maximal function with respect to a martingale sequence $(\EE_n)_n$ is a bounded map $L \log^2 L(\N) \to L_1(\N;\ell_\infty)$, and that the exponent $2$ on the log is optimal. By copying the computation of \cite{Stein1969LlogL} we have that 
  \begin{equation}
    \label{eq:MaximalLevelSets}
    \inf \left\{ \int_0^\infty \tau \big( q(\lambda)^\perp \big) \, d \lambda
                 \, : \, 
                   \begin{split} 
                     & q(\lambda) \in \Prj(\N) \mbox{ with }\\
                     & q(\lambda) \, f_n \, q(\lambda) \leq \lambda \, q(\lambda), \forall n \geq 1
                   \end{split}
                 \right\}
    \lesssim \| f \|_{L \log L(\N)}
  \end{equation}
   This leads to the following interesting observation: given a sequence $(f_n)_n \in L_1(\N; \ell_\infty)$ the quantity in the left hand side of \eqref{eq:MaximalLevelSets}
  is not comparable to the $L_1[\ell_\infty]$-norm. Indeed, if it were so, one would get that the $L_1[\ell_\infty]$-norm of $(\EE_n(f))_m$ would be bounded by the $L \log L$-norm of $f$. But that is not possible by \cite[Section 3]{Hu2009}.
  
  This gives an extra source of intuition regarding why it is difficult to obtain maximal $L_p$ inequalities from weak type ones (like the ones given by Cuculescus' and Yeadon's inequalities). It is not possible to reconstruct the $L_p[\ell_\infty]$-norm from the projections obtained in the weak type inequalities. 
\end{remark}

\subsection*{\bf Proof of the strong maximal theorem}

Now let us prove assertions \ref{itm:StrongIneq.1} and \ref{itm:StrongIneq.2} of Theorem \ref{thm:StrongIneq}. If $\lambda \leq (2e^2)^{1/\varepsilon}$, both inequalities trivially hold with  $q(\lambda) = 0$. When $\lambda > (2e^2)^{1/\varepsilon}$ we may assume that $\lambda = e^r$ for some positive integer $r \in \ZZ_+$. Let us first construct the projection $q(\lambda)$. Let $(f_n)_{n \geq 1}$ denote the positive martingale $(\EE_{n}^{[1]} \otimes \Id)(f)$ and consider Cuculescu's projections  $q_n^{[1]}(e^\ell)$ associated to $(f_n)_n$ at height $e^\ell$
\[
  q^{[1]}(e^\ell) = \bigwedge_{n \ge 1} q_{n}^{[1]}(e^\ell).
\]
Define the family of disjoint projections
\[
  \pi_{j}^{[1]} = \bigwedge_{\ell \geq j} q^{[1]}(e^\ell) - \bigwedge_{\ell \geq j-1} q^{[1]}(e^\ell)
  \quad \mbox{ for } \quad j \in \ZZ.
\]
Let $\rho_\lambda = \sum_{j \geq r} \pi_{j}^{[1]}$, where $r$ comes from the identity $\lambda = e^r$, and consider the (unbounded) operators
\begin{equation}
  \label{eq:ApproxMaximal}
  \Pi = \sum_{j \ge 1} e^j j^{1 + \varepsilon} \pi_{j}^{[1]}
  \quad \mbox{ and } \quad
  \Pi_\lambda = \rho_\lambda \Pi = \Pi \rho_\lambda.
\end{equation}
Next, construct the sequence of Cuculescu's projections $q_m^{[2]}(e^\ell)$ associated to the positive martingale
$(\Id \otimes \EE_{m}^{[2]})(\Pi_\lambda)$ at height $e^\ell$. As we did above, define the corresponding family of disjoint projections $\pi_{j}^{[2]}$ accordingly. We finally define 
\[
  q(\lambda) := \1 - \sum_{j \geq r} \pi_{j}^{[2]}.
\] 

\begin{proof}[Proof of inequality \ref{itm:StrongIneq.1}.]
We start noticing that
\begin{eqnarray*}
  q_\lambda \, f_{n,m} \, q_\lambda
    & = & q_\lambda \EE_{m}^{[2]} \big( \rho_\lambda \, f_n \, \rho_\lambda + \rho_\lambda \, f_n \, \rho_\lambda^\perp + \rho_\lambda^\perp f_n \rho_\lambda + \rho_\lambda^\perp \, f_n \, \rho_\lambda^\perp \big) q_\lambda \\
    & \leq & 2 q_\lambda \EE_{m}^{[2]} \big( \rho_\lambda \, f_n \, \rho_\lambda + \rho_\lambda^\perp f_n \rho_\lambda^\perp \big) \, q_\lambda  \, \leq \, 2 q_\lambda \EE_{m}^{[2]} \big( \rho_\lambda f_n \rho_\lambda \big) q_\lambda + 2 e^r q_\lambda. 
\end{eqnarray*} 
Indeed, to justify the last inequality recall that 
\begin{eqnarray}
  \rho_\lambda^\perp
    & = & \bigwedge_{\ell \in \ZZ} q^{[1]}(e^\ell) + \sum_{j < r} \pi_{j}^{[1]}, \label{eq:qInfty} \\
  \Big\| \bigwedge_{\ell \in \ZZ} q^{[1]}(e^\ell) f_n^{\frac12} \, \Big\|
    & = & \Big\| f_n^{\frac12} \bigwedge_{\ell \in \ZZ} q^{[1]}(e^\ell) \, \Big\| \, \leq \, \lim_{\ell \to -\infty} e^{\frac{\ell}{2}} q^{[1]}(e^\ell) = 0. \nonumber
\end{eqnarray}
In particular, we may estimate $\rho_\lambda^\perp \, f_n \, \rho_\lambda^\perp$ as follows 
\begin{eqnarray*}
  \big\| \rho_\lambda^\perp f_n \rho_\lambda^\perp \big\|_\M
    & = & \sup_{\|\alpha\|_{L_2(\M)} \le 1} \sum_{j,k < r} \big\langle \pi_{j}^{[1]} f_n \pi_{k}^{[1]} \alpha, \alpha \big\rangle \\
    & \leq & \sup_{\|\alpha\|_{L_2(\M)} \leq 1} \sum_{j,k < r} \big\| \pi_{j}^{[1]} f_n \pi_{k}^{[1]} \big\|_\M
             \big\| \pi_{j}^{[1]} \alpha \big\|_2 \big\| \pi_{k}^{[1]} \alpha \big\|_2 \\
    & \leq & \sup_{\|\alpha\|_{L_2(\M)} \leq 1} \sum_{j,k < r} e^{(j+k)/2} \big\| \pi_{j}^{[1]} \alpha \big\|_2 \big\| \pi_{k}^{[1]} \alpha \big\|_2 \, < \, e^r
\end{eqnarray*}
by H\"older's inequality. Next, to estimate $\rho_\lambda f_n \rho_\lambda = \Pi_\lambda^{\frac12} ( \Pi_\lambda^{-\frac12} f_n \Pi_\lambda^{-\frac12} ) \Pi_\lambda^{\frac12}$ we note that
\begin{eqnarray*}
  \big\| \Pi_\lambda^{-\frac12} f_n \Pi_\lambda^{-\frac12} \big\|_\M
    & = & \sup_{\|\alpha\|_{L_2(\M)} \leq 1} \sum_{j, k \geq r} \frac{e^{-(j+k)/2}}{(jk)^{(1+\varepsilon)/2}} \big\langle \pi_{j}^{[1]} f_n \pi_{k}^{[1]} \alpha, \alpha \big\rangle \\
    & \leq & \sup_{\|\alpha\|_{L_2(\M)} \leq 1} \Big( \sum_{j \geq r} \frac{\|\pi_{j}^{[1]} \alpha\|_2}{j^{\frac{1+\varepsilon}{2}}} \Big)^2 \ \leq \ \sum_{j \ge 1} \frac{1}{j^{1+\varepsilon}} \, =: \, A_{\varepsilon}.
\end{eqnarray*}
This implies that $\rho_\lambda \, f_n \, \rho_\lambda \leq A_\varepsilon \, \Pi_\lambda$ and therefore
\[
  q(\lambda) \, \EE_{m}^{[2]} \big( \rho_\lambda \, f_n \, \rho_\lambda \big) \, q(\lambda) \leq A_\varepsilon \, q(\lambda) \, \EE_{m}^{[2]} (\Pi_\lambda) \, q(\lambda) \leq A_\varepsilon \lambda \, q(\lambda).
\]
The last inequality follows from the definition of $q(\lambda)$ arguing as in \eqref{eq:qInfty} and the estimate after it. Altogether we have proved that $q(\lambda) \, f_{n,m} \, q(\lambda) \leq C_\varepsilon \lambda q(\lambda)$ for some constant $C_\varepsilon$ independent of $\lambda$ which diverges as $\varepsilon \to 0^+$. 
\end{proof}

\begin{proof}[Proof of inequality \ref{itm:StrongIneq.2}]
We have
\begin{eqnarray}
  \tau(\1 - q_\lambda)
    & = & \sum_{j > r-1} \tau(\pi_{j}^{[2]}) \, = \, \tau \Big( \1 - \bigwedge_{\ell \ge r-1} q_2(2^\ell) \Big) \nonumber\\
    & \leq & \sum_{\ell \ge r-1} \tau \big( \1 - q^{[2]}(e^\ell) \big) \, \leq \, \sum_{\ell \geq r-1} e^{-\ell}
    \tau (\Pi_\lambda) \quad \lesssim \, \frac1{\lambda} \quad \tau \big(\Pi_\lambda \big). \label{eq:asterisco}
\end{eqnarray}
We now expand the trace of $\Pi_\lambda/\lambda$ as follows
\[
  \tau(\Pi_\lambda/\lambda)
  =
  \int_{\RR_+} \tau \big( \1_{(\lambda s,\infty)}(\Pi_\lambda) \big) \, ds
  \leq \tau (\rho_\lambda) + \int_1^\infty \tau \big( \1_{(\lambda s,\infty)}(\Pi_\lambda) \big) \, ds.
\]
Arguing as above we see that $\tau(\rho_\lambda) \lesssim \lambda^{-1} \tau(f)$, which gives the first term in the sum defined by the right hand side of inequality \ref{itm:StrongIneq.2} in Theorem \ref{thm:StrongIneq}. It remains to prove the inequality  
\begin{equation}
  \label{Eq-}
  \int_1^\infty \tau \big( \1_{(\lambda s,\infty)}(\Pi_\lambda) \big) \, ds
  \, \leq \,
  C_{\varepsilon} \, \tau \left\{ \frac{f}{\lambda} \Big( 1 + \log_+ \Big[ \frac{f}{\lambda} \Big] \Big) \, \Big( \log_+ \Big[ \frac{f}{\lambda^{1-\epsilon}} \Big] \Big)^{1+\varepsilon}  \right\}.
\end{equation}
Since $\Pi_\lambda = \sum_{j \ge r} e^j j^{1 + \varepsilon} \pi_{j}^{[1]}$ and $\lambda = e^r$, we get
\[
  \1_{(\lambda s,\infty)}(\Pi_\lambda)
  \ = \!\!\!
  \sum_{\substack{ j \ge r \\ e^j j^{1 + \varepsilon} > \lambda s }} \pi_{j}^{[1]}
  \ = \ \begin{dcases}
          \quad \sum_{j \ge r} \, \pi_{j}^{[1]} & \mbox{ if } s < r^{1+\varepsilon},\\
          \sum_{e^j j^{1 + \varepsilon} > \lambda s} \pi_{j}^{[1]} & \mbox{ if } s > r^{1+\varepsilon}.
        \end{dcases}
\]
Let $j_0(s)$ be the biggest integer $j$ satisfying $e^j j^{1 + \varepsilon} \le \lambda s$. Arguing like in \eqref{eq:asterisco} and using Lemma \ref{lem:NewCuculescu} gives 
\begin{eqnarray*}
  \int_1^\infty \tau \big( \1_{(\lambda s,\infty)}(\Pi_\lambda) \big) \, ds
    & \leq & r^{1+\varepsilon} \tau(\rho_\lambda) + \int_{r^{1+\varepsilon}}^\infty \sum_{j > j_0(s)} \tau (\pi_{j}^{[1]}) \, ds \\
    & \leq & \big( \log \lambda \big)^{1+\varepsilon} \sum_{\ell \ge r} e^{-\ell} \tau \big( f \1_{(e^{\ell}/2,\infty)}(f) \big) \\
    & + & \int_{r^{1+\varepsilon}}^{e^r} \hskip-2pt \sum_{\ell \ge j_0(s)} e^{-\ell} \tau \big( f \1_{(e^{\ell}/2,\infty)}(f) \big) \, ds \\
    & + & \int_{e^r}^\infty \hskip2pt \sum_{\ell \ge j_0(s)} e^{-\ell} \tau \big( f \1_{(e^{\ell}/2,\infty)}(f) \big) \, ds \ =: \ \mathrm{A}_1 + \mathrm{A}_2 + \mathrm{A}_3.
\end{eqnarray*}
Notice that, for every $\varepsilon > 0$, we have that
\[
  \1_{( \frac{\lambda}{2}, \infty)}(f) = \1_{(\lambda, \infty)} \left( \Big( \frac{2 f}{\lambda^{1+\varepsilon}}\Big)^{\frac1{\varepsilon}} \right)
\]
and therefore, we trivially obtain that
\begin{eqnarray*}
  \mathrm{A}_1
    & \lesssim & \big( 1 + \log_+ \lambda \big)^{1+\varepsilon} \tau \Big( \frac{f}{\lambda} \1_{(\frac{\lambda}{2},\infty)}(f) \Big) \\
    & \lesssim & \tau \Big\{ \frac{f}{\lambda} \big( 1 + \log_+ \lambda \big)^{1+\varepsilon} \1_{(\frac{\lambda}{2},\infty)}(f) \Big\} \\
    & \lesssim & \Big(1 + \frac1{\varepsilon} \Big) \tau \Big\{ \frac{f}{\lambda} \Big( 1 + \log_+ \frac{f}{\lambda^{1-\varepsilon}} \Big)^{1+\varepsilon} \Big\}.
\end{eqnarray*}
Next, we use $\lambda > e^2$ to get $j_0(s) \ge 2$ so that 
\[
e^{-j_0(s)} \quad  =  \quad \frac{e  ( j_0(s)+1)^{1+\varepsilon}}{e^{j_0(s)+1} (j_0(s)+1)^{1+\varepsilon}} \lesssim \frac{ (\log (\lambda s))^{1+\varepsilon}}{\lambda s}.
%\\ \!\! & \lesssim & \!\! \frac{{\color{red} \big| \log [ e^{-(j_0(s)+1)} |j_0(s)+1|^{-(1+\varepsilon)}] \big|^{1+\varepsilon}}}{e^{j_0(s)+1} (j_0(s)+1)^{1+\varepsilon}} \quad \lesssim \quad \Big(\frac{3}{2}\Big)^{1 + \varepsilon} \, e \, \frac{{\color{red} |\log (\lambda s)|^{1+\varepsilon}}}{\lambda s}.   
\]
for $s > r^{1+\varepsilon}$. Since $\log (\lambda s) \le 2 \max \{\log \lambda, \log s\}$, we get 
\[
  \sum_{\ell \ge j_0(s)} e^{-\ell} \tau \big( f \1_{(e^{\ell}/2,\infty)}(f) \big) 
  \, \lesssim \,
  \begin{dcases}
    \frac{\log^{1+\varepsilon} (\lambda)}{\lambda s} \tau \Big( f \1_{(\frac{\lambda s}{\log^{1+\varepsilon} (\lambda)},\infty)}(10 f) \Big) & \mbox{if } \lambda \geq s, \\
    \frac{\log^{1+\varepsilon} (s)}{\lambda s} \tau \Big( f \1_{(\frac{\lambda s}{\log^{1+\varepsilon} (s)},\infty)}(10 f) \Big) & \mbox{if }  s > \lambda.
  \end{dcases}.
\]
This gives the following estimate for $\mathrm{A}_2$, arguing as for $A_1$:
\begin{eqnarray*}
  \mathrm{A}_2
     & \leq & \frac{ \log^{1+\varepsilon} (\lambda)}{\lambda}  \int_{ \log^{1+\varepsilon} (\lambda)}^\lambda \tau \Big\{ f \1_{(\frac{\lambda s}{\log^{1+\varepsilon} (\lambda)},\infty)}(10 f) \Big\} \, \frac{ds}{s} \\
 \!\! & \le & \!\! \frac{ \log^{1+\varepsilon} (\lambda)}{\lambda} \, \tau \Big\{ f \log \Big[ \frac{4 e f}{\lambda}  \log^{1+\varepsilon} (\lambda) \Big] \1_{(\frac{\lambda s}{\log^{1+\varepsilon} (\lambda)},\infty)}(10 f) \Big\} \\ 
 & \lesssim_{\varepsilon} & \tau \bigg\{ \frac{f}{\lambda}  \Big( 1 + \log_+ \Big( \frac{f}{\lambda} \Big)\Big) \Big( 1 + \log_+ \Big( \frac{f}{\lambda^{1-\varepsilon}} \Big) \Big)^{1+\varepsilon} \bigg\}\\  
\end{eqnarray*}
Similarly, we get
\[
  \mathrm{A}_3 \le \frac{1}{\lambda} \int_\lambda^\infty \tau \Big( f \1_{(\frac{\lambda s}{\log^{1+\varepsilon} (s)},\infty)}(10 f) \Big) \, \frac{\log^{1+\varepsilon}(s)}{s} ds.
\]
Next, given $s > \lambda > e^{16}$ and $\varepsilon < 1$, observe that
\[
\frac{s}{\log^{1+\varepsilon}(s)} < A \ \Rightarrow \ s < A \log^{1+\varepsilon}(s) < 4 A \log^{1+\varepsilon} (A).
\]
Therefore, approximating the operator $10f/\lambda$ by a positive combination of pairwise disjoint projections $\sum_j a_j P_j$ we may rewrite the above integral as follows and apply Fubini to get the  following estimate 
\begin{eqnarray*}
  \mathrm{A}_3
    & \le & \tau \Big( \frac{f}{\lambda} \int_\lambda^\infty \sum_{a_j > \frac{s}{\log^{1+\varepsilon}(s)}} P_j \, \frac{\log^{1+\varepsilon}(s)}{s} ds \Big) \\ & \le & \tau \Big( \frac{f}{\lambda} \sum_{a_j > \frac{\lambda}{\log^{1+\varepsilon}(\lambda)}} \Big[ \int_\lambda^{4 a_j \log^{1+\varepsilon}(a_j)} \frac{\log^{1+\varepsilon}(s)}{s} ds \Big] P_j \Big) \\
    & \le & \tau \Big( \frac{f}{\lambda} \sum_{a_j > \frac{\lambda}{\log^{1+\varepsilon}(\lambda)}} \log^{2+\varepsilon} \big( 4 a_j \log^{1+\varepsilon}(a_j) \big) P_j \Big) \\ [5pt]
    & \le & \tau \Big( \frac{f}{\lambda} \1_{(\frac{\lambda}{\log^{1+\varepsilon}(\lambda)},\infty)}\Big(\frac{10 f}{\lambda} \Big) \log^{2+\varepsilon} \Big[ \frac{40 f}{\lambda} \log^{1+\varepsilon} \Big( \frac{10 f}{\lambda} \Big) \Big] \Big) \\ [10pt]
    & \lesssim & \tau \bigg( \frac{f}{\lambda} \Big( \log_+\Big[ \frac{f}{\lambda} \Big] \Big)^{2+\varepsilon} \bigg) 
\end{eqnarray*}
\end{proof}

Several remarks are in order.

\begin{remark}
  \label{rmk:LongRemark}
  \
  \begin{enumerate}[leftmargin=1.2cm, label={\bf R.\arabic*}, ref={\bf R.\arabic*}]
    \item \label{itm:RemarkGuzman}
    The proof above can be understood as a noncommutative generalization of a result of de Guzm\'an. Indeed, in \cite{Guzman1972ProductBases}, it was proven that if two bases of Borel sets $\mathfrak{B}_i \subset \B(\RR^{d_i})$ with finite measure satisfy that their associated maximal operators over $\RR^{d_i}$
    \[
      (M_{i}f)(x) = \sup_{E \in \mathfrak{B}_i} \dashint_{E} |f(x)| \, d x
    \]
    are of weak type $(\Phi_i, \Phi_i)$, i.e
    \[
      \big| \big\{
              x \in \RR^{d_i} \, : \,
              (M_i f)(x) > \lambda
            \big\} \big|
      \, \lesssim \,
      \int_{\RR^{d_i}} \Phi_i \left( \frac{|f(x)|}{\lambda}\right) \, d x
    \]
    Then, the tensor product of both maximal operators $M_1 \otimes M_2$
    is of weak type $(\Psi, \Psi)$ over $\RR^{d_1 + d_2}$, were $\Psi$ is given by a combination
    of $\Phi_1$ and $\Phi_2$. 
  
    When specialized to the commutative case, our proof above gives what is perhaps the simplest version de Guzm\'an's and Cordoba/Fefferman inequality \ref{eq:classicalCF} in the martingale case.  
    Our proof also highlights the different behavior of the Cuculescu projections compared with the level sets of the classical maximal. Indeed, in the
    classical case we do have that
    \[
      \Pi \,  = \, \sum_{j \geq 1} e^j \pi_{j}^{[1]} \, = \, \sum_{j \geq 1} e^j \, \1_{\{e^{j} < M_1(f) \leq e^{j + 1}\}}
    \]
    is comparable to the maximal function $(M_1 \otimes \Id)(f)$. While,
    see Remark \ref{rmk:SteinLlogL}, the same cannot hold in the noncommutative
    case. Nevertheless, the modified $\Pi$ of \eqref{eq:ApproxMaximal} with the
    extra term $j^{1 + \varepsilon}$ can be used to control the behavior of the maximal.

  \item \label{itm:OptimalExp}
  The proof above can be modified by introducing other terms of
  summable inverse in the definition of $\Pi$, thus giving inequalities like \ref{itm:StrongIneq.2}
  for more complicated Orlicz spaces like
  \[
    L \log^2 L \log \log^{1 + \varepsilon} L(\M).
  \]
  We have chosen not to perform our calculation for that particular space because our inequality recovers   what we believe is the optimal case, see Conjecture \ref{cjn:StrongIneqC}, when $\varepsilon \to 0^+$, while the space above would not.
  
  \item \label{itm:Guixiang}
  A similar result to Theorem \ref{thm:StrongIneq} has been obtained in the literature, 
  see \cite[Theorem 3.5]{HongSun2018}. Nevertheless, there are key differences between the two approaches. Although the results \cite{HongSun2018} were formulated in the context of multiparametric ergodic means, they can be easily transferred to the context of martingales. Indeed, following their scheme of proof, we can use an atomic decomposition of the $L \log^\alpha L$ spaces due originally to Tao \cite{Tao2001Extrapolation}
  in the commutative setting, to obtain that, for every $s \geq 0$
  \[
    \big\| (\EE_n)_n: L \log^{2 + s} L(\M) \to L \log^s(\M) \big\| < \infty.
  \]
  Then, starting with an element in $L \log^{2} L(\M)$, we can bound $\EE_{n}^{[1]} \otimes \EE_{m}^{[2]}$
  by iterating the technique above, applying it first to $(\Id \otimes \EE_m^{[2]})(f)$ for positive $f$
  to obtain an element
  \[
    (\Id \otimes \EE_m^{[2]})(f) \leq g \quad \mbox{ with } \quad \| g \|_1 \leq \| f \|_{L \log^2 L}
  \]
  since $g$ is in $L_1$, we can apply the maximal weak type $(1,1)$ for the family $(\EE_{n}^{[1]} \otimes \Id)_{n}$.
  As a result, we obtain that for every $f \in L \log^{2} L(\M)_+$ and $\lambda > 0$, there is a projection
  $q(\lambda) \in \Prj(\M)$ satisfying that
  \begin{itemize}[leftmargin=1.5cm]
    \item $\displaystyle{q(\lambda) \, \big(\EE_n^{[1]} \otimes \EE_m^{[2]}\big)(f) \, q(\lambda)
    \, \, \leq \, \, \lambda q(\lambda)}$.
    \item $\displaystyle{ \tau \big( q(\lambda)^\perp \big)  \, \leq \, \frac{\| f \|_{L \log^2 L}}{\lambda}}$
  \end{itemize}
  This results generalizes to iterated products by starting with an element in $L \log^{2(d-1)} L$.
  This result can be understood as a bound
  \[
    (\EE_n^{[1]} \otimes \EE_m^{[2]})_{n,m}: L \log^2 L(\M)_+ \to L_{1,\infty}(\M; \ell_\infty)_+.
  \] 
  Nevertheless, it is important to notice that neither of the results is stronger than the other. In one direction, we cannot deduce the bound above without an extra $\varepsilon$ in the exponent. In the other, our result gives projections with a decay in $\lambda$ that is strictly faster that $O(\lambda^{-1})$. Since the proof of \cite[Theorem 3.5]{HongSun2018} requires going through $L_1[\ell_\infty]$, by Remark \ref{rmk:SteinLlogL} one must loose some factor with respect to the optimal size of the projections $q(\lambda)$. We also recall that having noncommutative maximal weak type $(\Phi, \Phi)$ estimates, with $\Phi(r) = r (1 + \log_+^{2 + \varepsilon}(r))$ may be important for the purpose of extending real interpolation results beteween the weak Orlicz types.
  \end{enumerate}
\end{remark}

\textbf{Acknowledgement.} The authors are thankful to Simeng Wang for pointing out the existence of \cite{HongSun2018} and for providing a reference for the asymmetric convergence results in \cite{Litvinov2012Ergodic}. Those fruitful discussions took place during ICMAT's School I of the <<Thematic Research Program: Operator Algebras, Groups and
Applications to Quantum Information>> in March 2019. The authors are also indebted to Guixiang Hong for the comments provided at Harbin Institute of Technology, China. 

%\cite{Tao2015failure}

\begin{small}
  \bibliographystyle{acm}
  \bibliography{../bibliography/bibliography}
\end{small}

\vspace{50pt}

\hfill \noindent \textbf{Jose M. Conde-Alonso} \\
\null \hfill UAM - Departamento de Matem\'aticas \\ 
\null \hfill 7 Francisco Tom\'as y Valiente, 28049 Madrid, Spain
\\ \null \hfill\texttt{jose.conde@uam.es}

\

\hfill \noindent \textbf{Adri\'an M. Gonz\'alez-P\'erez} \\
\null \hfill K U Leuven - Departement wiskunde \\ 
\null \hfill 200B Celestijnenlaan, 3001 Leuven, Belgium 
\\ \null \hfill\texttt{adrian.gonzalezperez@kuleuven.be}

\

\hfill \noindent \textbf{Javier Parcet} \\
\null \hfill Consejo Superior de Investigaciones Cient\'ificas - ICMAT\\ 
\null \hfill 23 Nicol\'as Cabrera, 28049 Madrid, Spain 
\\ \null \hfill\texttt{parcet@icmat.es}

\end{document}